\documentclass[11pt]{article}
\usepackage{framed}

\usepackage{hyperref}
\usepackage{graphicx}
\usepackage[titletoc,title]{appendix}
\usepackage[applemac]{inputenc}
\usepackage{enumerate}
\usepackage{comment}
\usepackage[dvipsnames]{xcolor}

\usepackage{amsmath,amssymb,wasysym}
\usepackage[english]{babel}
\usepackage{a4wide}
\usepackage[sans]{dsfont}
\usepackage{amsfonts}
\usepackage{amsthm}
\usepackage[mathscr]{euscript}
\usepackage{mathrsfs}
\usepackage{latexsym}

\newtheorem{theorem}{Theorem}[section]
\newtheorem{lemma}[theorem]{Lemma}
\newtheorem{proposition}[theorem]{Proposition}
\newtheorem{corollary}[theorem]{Corollary}
\newtheorem{conjecture}[theorem]{Conjecture}
\theoremstyle{definition}
\newtheorem{definition}[theorem]{Definition}
\newtheorem{example}[theorem]{Example}

\newtheorem{rem}[theorem]{Remark}

\newcommand{\p}{^}

\newcommand{\pb}[2]{
	\ensuremath{\langle #1,#2 \rangle}}

\newcommand{\m}{{\bf m}}

\newcommand{\rmd}{d}
\newcommand{\rme}{\mathrm{e}}

\newcommand{\Xscr}{\mathscr{X}}

\newcommand{\Nscr}{\mathscr{N}}

 \newcommand{\al}{\alpha}

\newcommand{\gm}{\gamma}

\newcommand{\ep}{\varepsilon}
\newcommand{\lm}{\lambda}
\newcommand{\Lm}{\Lambda}
\newcommand{\sig}{\sigma}

\newcommand{\Om}{\Omega}

\newcommand{\cB}{\mathcal{B}}
\newcommand{\cS}{\mathcal{S}}

\newcommand{\cM}{\mathcal{M}}
\newcommand{\cN}{\mathcal{N}}

\newcommand{\Ebb}{\mathbb{E}}
\newcommand{\EE}{\mathbb{E}}

\newcommand{\PP}{\mathbb{P}}
\newcommand{\Pbb}{\mathbb{P}}

\newcommand{\Rbb}{\mathbb{R}}

\newcommand{\R}{\mathbb{R}}

\newcommand{\bi}{\begin{itemize}}
\newcommand{\ei}{\end{itemize}}

\newcommand{\be}{\begin{enumerate}}
\newcommand{\ee}{\end{enumerate}}
\newcommand{\beq}{\begin{equation}}
\newcommand{\eeq}{\end{equation}}
\newcommand{\beqs}{\begin{equation*}}
\newcommand{\eeqs}{\end{equation*}}
\newcommand{\beqa}{\begin{eqnarray}}
\newcommand{\eeqa}{\end{eqnarray}}
\newcommand{\beqas}{\begin{eqnarray*}}
\newcommand{\eeqas}{\end{eqnarray*}}

 \newcommand{\aPP}[2]{\ensuremath{\langle #1,#2 \rangle}}

\newcommand{\equa}{\begin{eqnarray*}}
\newcommand{\tion}{\end{eqnarray*}}
\newcommand{\equal}{\begin{eqnarray}}
\newcommand{\tionl}{\end{eqnarray}}
\newcommand{\one}{\mathds{1}}

\newcommand{\N}{\mathbb{N}}

\newcommand{\sn}{\eta_{\scalebox{0.5}{$ N$}}}

\newcommand{\ssn}{\eta_{\scalebox{0.4}{$ N$}}}

\makeatletter
\def\timenow{\@tempcnta\time
\@tempcntb\@tempcnta
\divide\@tempcntb60
\ifnum10>\@tempcntb0\fi\number\@tempcntb
:\multiply\@tempcntb60
\advance\@tempcnta-\@tempcntb
\ifnum10>\@tempcnta0\fi\number\@tempcnta}
\makeatother

\newcommand{\new}{\color{blue}}

\title{Product  formulas for multiple stochastic integrals \\ associated with L\'evy processes}

\author{Paolo Di Tella$^1$   \hspace{0.9em} Christel Geiss$^2$   \hspace{0.9em}  Alexander Steinicke$^{3,2}$  \vspace*{-0.3em}\\  
 {  \tiny{ \hspace{-2.8em} \tt paolo.di\_tella{\rm@}tu-dresden.de}   \hspace{2.0em}{\tt christel.geiss{\rm@}jyu.fi}     \hspace{2.em}  {\tt alexander.steinicke{\rm@}unileoben.ac.at}} \\\\
\small  \today}

\date{}

\begin{document}
\parindent 0pt

\maketitle
\begin{abstract}
 In the present paper, we obtain an explicit product formula for products of multiple integrals w.r.t. a random measure associated with a L\'evy process. As a building block, we use  a representation formula for products of martingales from a compensated-covariation stable family. This enables us to consider L\'evy processes with both  jump and Gaussian part. It is well known that for multiple integrals w.r.t.~the Brownian motion such product formulas exist without further integrability conditions on the kernels. However, if a jump part is present, this is, in general, false. Therefore, we provide here sufficient conditions on the kernels which allow us to establish product formulas. As an application, we obtain explicit expressions for the expectation of products of iterated integrals, as well as for the moments and the cumulants for stochastic integrals w.r.t.~the random measure. Based on these expressions, we show a central limit theorem for the long time behaviour of a class of stochastic integrals. Finally, we provide methods to calculate the number of summands in the product formula.  
\end{abstract}

\vspace{1em}
{\noindent \textit{Keywords:} Product  formulas for stochastic integrals, moment formulas, multiple integrals, L\'evy processes, iterated integrals, central limit theorem, compensated-covariation stable families\\
\noindent \textit{Mathematics Subject Classification:} 60H05, 60G51, 60G44, 60F05
}
{
\footnotetext[1]{Institute for Mathematical Stochastics, TU Dresden, Germany.   }
\footnotetext[2]{Department of Mathematics and Statistics, University of Jyvaskyla, Finland. }
\footnotetext[3]{Department of Mathematics and Information Technology, University of Leoben, Austria.   }}

\section{Introduction  }\label{sec:intro}
In the present work we consider multiple Brownian-Poisson stochastic integrals, that is, multiple stochastic integrals generated by  a random measure that  exhibits both a Gaussian and a Poisson part. This kind of random measures are associated  with L\'evy processes. We  show in Theorem \ref{L2condition-thm} a compact and accessible product formula for the product  of $N\geq2$ multiple Brownian-Poisson stochastic integrals,  extending previous results obtained for $N=2$ (see for instance \cite{S84}, \cite{Kab}, \cite{LeeShih04}, \cite{N06}, \cite{PT11}, \cite{Major14}, \cite{DoPecc18}). 
   In such formulas, products of multiple stochastic integrals are expressed as a sum of multiple stochastic integrals. 
 
To state our product formula, a crucial step is the definition of the \emph{star operator} $\bigstar^l_{l^o}$ (see Definition \ref{contractions_and_identifications}). It generalizes the identification and contraction rules introduced e.g.~in \cite{DoPecc18A} for {$N=2$, to an arbitrary number $N\geq2$ of multiple integrals}.   \smallskip

Multiple stochastic integrals play an important role in stochastic analysis {because of their use for the} chaos expansion and for the definition of the  Malliavin derivative. Moreover, {their products} find applications in certain limit theorems and numerical simulations \cite{PT11}. 

The classic example of such a product formula can be found  e.g.~in  Nualart \cite[Proposition 1.1.3]{N06}, where for symmetric functions $f\in L^2([0,T]^n),\, g\in  L^2([0,T]^m)$, the product of two multiple integrals w.r.t.\ the Gaussian measure generated by the  Brownian motion is given by  
\equal \label{classical-prod}
    I_n(f) I_m(g) = \sum_{r=0}^{n\wedge m} r! \binom{n}{r} \binom{m}{r} I_{n+m-2r}(f\otimes_r g).
\tionl
Here, $f\otimes_r g$ denotes an operator which 'contracts' $r$ variables of $f$ and  of $g$ (i.e., they are identified and integrated out).  For the extension of \eqref{classical-prod} to an arbitrary $N\geq2$, we refer to Major \cite[Theorem 10.2]{Major13}. 
It is well known that for multiple Brownian integrals the product formula always holds for any $N\geq2$, with no need of further integrability conditions on the kernels. Contrarily, for multiple Poisson integrals, i.e., for multiple stochastic integrals generated by a Poisson random measure, this is not true in general and, as a minimal requirement for the product formula, one has to ensure that the
product of multiple Poisson stochastic integrals is again square integrable. Even  in the elementary case of a compensated Poisson process, the latter condition may be violated:  Let $\mathcal{N}$ be a standard Poisson process and set $L_t = \mathcal{N}_t - t$. For any $f\in L^2([0,T])$, by It\^o's formula, we have

\equa
     \left (\int_0^T f(t) \rmd  L_t \right)^2 = 2 \int_0^Tf(t)\left(\int_0^{t-} f(s) \rmd L_s\right)  \rmd L_t +  \int_0^T f(t)^2\rmd L_t 
     + \int_0^T f(t)^2 \rmd t,
\tion 
which is square integrable if and only if $f\in L\p4([0,T])$. Interpreting the left-hand side as the product of two  multiple stochastic integrals  of order one, we see that this is not  square integrable, and hence does not have a chaos decomposition, unless $f\in L\p4([0,T])$ .  \smallskip

 The first product formula was established by It\^o \cite{I51} in 1951 and corresponds to \eqref{classical-prod} for an arbitrary $n$ but for $m=1$. Kabanov \cite{Kab} extended this formula to multiple  Poisson stochastic integrals (keeping $m=1$). Surgailis \cite{S84} proved a product formula for $N\geq2$ multiple Poisson
integrals in 1984 and asked whether the square integrability of their product is a sufficient condition for the product formula to hold. Recently, in \cite[Theorem 2.2]{DoPecc18A}, D\"obler and Peccati gave a positive answer for the case of two multiple Poisson integrals. However, the case of an arbitrary number of factors seems to be still open and, in the present paper, we formulated this problem in Conjecture \ref{conjecture}. Lee and Shih were the first authors who in \cite{LeeShih04} unified the Brownian and the Poisson case considering product formulas for Brownian-Poisson multiple integrals.
 \smallskip
 
 Product formulas for multiple integrals have been obtained in the literature by several methods. For example, in \cite{Major14}, \cite{S84}, \cite{PT11} \textit{diagram formulae} are used. In \cite{N06} and \cite{Last16}, the formulas are obtained by induction on the order of the multiple integrals. 
In \cite{Meyer93} and \cite{LeeShih04}, the relation between the boson Fock space and the representation of the stochastic exponential of $I_1(f)$ is exploited. In \cite{Ag20}, the properties of the  stochastic exponential are used to derive a formal expression for the product formula for $N\geq2$ multiple Poisson stochastic integrals. Some papers, as e.g.~\cite{LPST14} or \cite{Bogdan-etal}, obtain explicit moment and cumulant formulas for multiple Poisson integrals relying on Mecke's formula rather than on product formulas. In the present paper, as done in \cite{DTG19} for a different context, we use the properties of compensated-covariation stable families of martingales. This allows us to directly show a product formula for $N\geq2$ multiple Brownian-Poisson stochastic integrals associated with arbitrary L\'evy processes. Obviously, a moment formula can be immediately obtained by taking the expectation in our product formula.\smallskip

 The present paper has the following structure: 
Section \ref{sec:mu.in} is devoted to multiple Brownian-Poisson stochastic integrals. Section \ref{sec:com.cov.st} recalls a representation formula for products of compen\-sated-covariation stable families of martingales that we apply in Section \ref{sec:fi.pro.for} to obtain a first product formula for iterated integrals. This product formula  is then    transformed  into one  for multiple integrals  and extended to the general setting of Theorem \ref{L2condition-thm}. The latter is applied in Section \ref{sec:appl} to obtain  moment and cumulant formulas for stochastic integrals. These formulas are then exploited to show a central limit theorem that describes the long time behaviour of multiple integrals of order one for a class of integrands. 
Finally, in Section \ref{sec:DSSn} we provide several methods of  algorithmic interest to count the number of summands  contained in the product formula from Section \ref{sec:fi.pro.for}.  

\section{Multiple  integrals}\label{sec:mu.in}

Let $L=\left(L_t\right)_{t\geq0}$ be a c\`adl\`ag one-dimensional L\'evy process with   characteristic triplet  $(\gm,\sig\p2,\nu)$   on a complete probability space $(\Omega,\mathcal{F},\mathbb{P})$. 
The augmented natural filtration of $L$  will be denoted by
$\left({\mathcal{F}_t}\right)_{t\geq0}$. We now fix $T>0$, restrict our analysis to $[0,T]$ and assume that  $\mathcal{F}=\mathcal{F}_T$. 

 The L\'evy-It\^o decomposition of {$L$ reads}
\equa
L_t = \gamma t + \sigma W_t   +  \int_{{(0,t]}\times \{ |x|\le1\}} x\tilde{\mathcal{N}}(\rmd s,\rmd x) +  \int_{{(0,t]}\times \{ |x|> 1\}} x  \mathcal{N}(\rmd s,\rmd x),
\tion
where $\gamma\in \R, \sigma\geq 0$, $W$ is a  standard Brownian motion and $\mathcal{N}$ ($\tilde{\mathcal{N}}$) is the (compensated) Poisson random measure corresponding to $L$. 
The random measure $M$ given by
$$   M(\rmd t, \rmd x) = \sigma \rmd W_t \delta_0(\rmd x) + \tilde{\mathcal{N}}(\rmd t,\rmd x), $$
where $\delta_0$ {denotes} the Dirac measure concentrated in zero, is used to define the following {\it iterated integrals}:  We set $\Rbb_0 := \Rbb \setminus \{0\}$ and
       $$\m(\rmd t,\rmd x) := \rmd t (\sigma^2 \delta_0(\rmd x) + \nu(\rmd x)).$$ 
        For { $k \geq1$} we put 
$$f\in L^p_k:=L^p(([0,T]\times\Rbb)^k, \mathcal{B}( ([0,T]\times\Rbb)^k),   \m^{\otimes k}), \quad \text{for } \,\, p \ge 1,$$
and denote {$\|\cdot\|_{ L^p_k}$ the $L^p_k$-norm of $f$}.
We denote by  $\hat f((t_1,x_1),...,(t_k,x_k))$  the symmetrisation of the kernel
$f((t_1,x_1),...,(t_k,x_k))$ w.r.t.the $k$ pairs of variables:

\begin{equation}\label{eq:symm}
\hat f((t_1,x_1),...,(t_k,x_k)) =  \frac{1}{k!} \sum_{\pi \in {\tt S}_k}  f((t_{\pi(1)},x_{\pi(1)}),...,(t_{\pi(k)},x_{\pi(k)})),
\end{equation}

 where ${\tt S}_k$ is the set of all permutations $\pi$ of $\{1,...,k\}.$  {We  introduce the \emph{iterated integrals} as  follows:}
\equa
J_0 &:=& \text{the identity map on} \,\, \R \\
J_1( f)_T &:=&  \int_{(0,T]\times\R}  f(t,x) M(\rmd t,\rmd x),  \quad f\in  L^2_1,  \\
J_2(\hat f)_T &:=&  \int_{(0,T]\times\R} \int_{(0,t_2)\times\R} \hat f((t_1,x_1),(t_2,x_2)) M(\rmd t_1,\rmd x_1) M(\rmd t_2,\rmd x_2),  \quad f\in  L^2_2,  \\
J_k(\hat f)_T &:= &  \int_{(0,T]\times\R} J_{k-1}( \hat f(\dots,(t_k,x_k) ))_{t_k-} \,M(\rmd t_k,\rmd x_k), \quad f\in  L^2_k.
\tion

To define  {\it multiple integrals} w.r.t.~$M$, one considers  the linear space ${\cal E}_k$ of simple functions of the form 
\equal \label{L0m}
g(z_1,...,z_k) = \sum_{j_1,...,j_k=1}^n  a_{j_1,...,j_k}   \one_{A_{j_1}\times ...\times A_{j_k}}(z_1,...,z_k), \quad  z_i:=(t_j,x_j),
\tionl
where the $A_1,...,A_n \in\mathcal{B}( [0,T]\times \R)$ are disjoint and such that $\m(A_j)<\infty$ for all $j=1,...,n.$  Moreover, $a_{j_1,...,j_k}$ is zero whenever two or more of the indices
${j_1,...,j_k}$ coincide.  {We then put}
$$ I_k(g)_T :=  \sum_{j_1,...,j_k=1}^n  a_{j_1,...,j_k} M(A_{j_1})\cdots M(A_{j_k}). $$

It holds that $\EE M(A_j) =0$ and $\EE M(A_j)^2 = \m(A_j),$  for  $j=1,...,n,$ and  the $M(A_1),\dots, M(A_n)$ are independent. { In the next lemma we recall some properties of $J_k$ and $I_k.$}

\begin{lemma} \label{iterated-and-multiple} {  Let $k,m\in \N$ and $f \in   L_k^2$, $g \in   L_m^2$. We then have:}
\begin{enumerate}[(i)]
\item   \label{extension}  The space ${\cal E}_k$ is  dense in $L^2_k$, and $I_k:{\cal E}_k \to  L^2(\Pbb)$ is a linear operator. The map $I_k$  has a unique extension  $I_k:L^2_k \to  L^2(\Pbb)$.
{\item  \label{L2-estimate} $\EE I_k(f)_T=0$, $\| I_k(f)_T\|_{L^2(\Pbb)} =  k!  \| \hat f \|_{L_k^2} \le  k!  \|  f \|_{L_k^2},$ and $\EE I_k(f)_T I_m(g)_T =0$ if $m\neq k$.
\item \label{with-hat-or-not} $ I_k(\hat f)_T = I_k(f)_T$ and $ I_k(\hat f)_T = k! J_k(\hat f)_T$.}
\end{enumerate}
\end{lemma}

\begin{proof} \eqref{extension}  By  \cite[Theorem 2.1]{I51} we have that   ${\cal E}_k$ is  dense in $L^2_k.$
In \cite{I51} it is also shown that \eqref{L2-estimate} holds for all $f \in {\cal E}_k$, and therefore, by linearity and continuity, $I_k$ has a unique extension to $L^2_k$. Also \eqref{with-hat-or-not} is first shown for all $f \in {\cal E}_k$ and then extended to $L^2_k$.
For convenience and later use we recall the proof of the relation between $I_k$ and $J_k$ in \eqref{with-hat-or-not}. We follow \cite{LeeShih04}  and  start 

assuming $g (z_1,...,z_k) =  \sum_{j=1}^N a_j  \one_{A_1^j \times ... \times     A_k^j} (z_1,...,z_k) $
for measurable and  pairwise disjoint $A_1^j, ...,  A_k^j $ satisfying $\m(A^j_l)<\infty$ for  $l=1,...,k$. Moreover, we assume $g(z_1,...,z_k) =0$ for $(z_1,...,z_k) \notin  \Delta_{\,T}(k)$, where 
\equal \label{Delta}
\Delta_{\,T}(k):  =\{((t_1,x_1),\ldots,(t_k,x_{k})): \ 0 \le t_1  <\ldots  < t_k < T,  \,x_1,...,x_k \in \Rbb\}
\tionl
so that, if   $(t_1,x_1) \in A^j_i$ and $(t_2,x_2) \in A^j_l$, then  $t_1<t_2$  for  $i<l.$  By \eqref{eq:symm} and the definition of $g$, we get 
 \equa
I_k(\hat g)_T &=& I_k \left (\frac{1}{k!} \sum_{\pi \in {\tt S}_k}  \sum_{j=1}^N a_j \one_{A_{\pi(1)}^j \times ... \times     A_{\pi(k)}^j} \right )_T\\
&=& \frac{1}{k!} \sum_{\pi \in {\tt S}_k}  \sum_{j=1}^N a_j M(A_{\pi(1)}^j)\cdots M (A_{\pi(k)}^j)     
=  \sum_{j=1}^N a_j M(A_1^j)\cdots M (A_k^j)  = I_k(g)_T.     
\tion
On the other hand, the definition of the iterated integrals yields $$J_k(\hat g)_T =\frac{1}{k!} \sum_{j=1}^N a_j M(A_1^j)\cdots M (A_k^j) = \frac{1}{k!}I_k(g)_T.$$
This extends to symmetric kernels $\hat  f  \in L_k^2,$ 
hence the iterated integrals and the multiple integrals are related by $ I_k(\hat f)_T = k! J_k(\hat f)_T. $
\end{proof}    
We recall It\^o's chaos expansion result:

\begin{theorem}[\cite{I56}] There exists for any 
$\xi \in L_2(\Om,\mathcal{F},\mathbb{P})$  a unique  chaos expansion 
\equa
\xi=\sum_{k=0}^\infty I_k(\hat f)_T,
\tion
where  $\hat f \in L_k^2,$ and $I_0$ is the identity map on $\R.$
\end{theorem}
Our  aim is  to show a product formula, that means we want to  find  conditions on the $(f^{(j)})_{j=1}^N,$ where $f^{(j)} \in    L_{k_j}^2,$   such  that the {\it product}
$$ \prod_{j=1}^N  I_{k_j}(\hat f^{(j)})_T $$
can be written as a  {\it  sum} of square integrable multiple integrals.

\section{{ Compensated-covariation stability and martingale products}}\label{sec:com.cov.st}
 To state a representation result for products of martingales, we first need to recall the concept of compensated covariation stable families (see \cite[Definition 3.1]{DTE15}):  Let $\Lm$ be an arbitrary parameter set and $\Xscr=\{X\p\al,\ \al\in\Lm\}$  a family of square integrable martingales. Set 
 $$X\p{\al_{1:2}}:=[X\p{\al_1},X\p{\al_2}]-\aPP{X\p{\al_1}}{X\p{\al_2}}, \quad \al_1,\al_2\in\Lm.$$ Then $\Xscr$ is called compensated-covariation stable if  $X\p{\al_{1:2}}\in\Xscr$ holds for any $\al_1,\al_2\in\Lm$. Let $\Xscr$ be  such a family. For $\al_1,\ldots,\al_k\in\Lm$ we can recursively define
 \equal \label{iteration-comp-cov}
 X^{\al_{j_{1:k}}}:= [X^{\al_{j_{1:(k-1)}}},X^{\al_{j_k}}]- \langle X^{\al_{j_{1:(k-1)}}},X^{\al_{j_k}}\rangle. 
\tionl
For a compensated-covariation stable family $\Xscr$  the following representation formula holds:
\begin{proposition}[{\cite[Proposition 3.3]{DTE15}}]\label{prop:rep.pol}
Let  $\{X\p{\al},\ \al\in\Lm\},$  be a compensated-covariation stable family of quasi-left continuous square integrable martingales. For every $N\geq1$ and $\al_1,\ldots,\al_N\in \Lm$, we have
\equal \label{eq:rep.pol}
 \prod_{i=1}\p N X\p{\al_{i}}_t&=&
\sum_{i=1}\p N\sum_{{1\leq j_1<\ldots<j_i\leq N}} \int_0^t \Bigl(  \prod_{\begin{subarray}{c}k=1\\k\neq j_1,\ldots,j_i\end{subarray}}\p N X\p{\al_k}_{s-}\Bigr)
d X\p{\al_{j_{1:i}}}_s \notag
\\&&+
\sum_{i=2}\p N\sum_{1\leq j_1<\ldots<j_i\leq N} \int_0^t \Bigl(\prod_{\begin{subarray}{c}k=1\\k\neq j_1,\ldots,j_i\end{subarray}}\p N X\p{\al_k}_{s-}\Bigr)d
\langle X\p{\al_{j_{1: {i-1}}}}{X\p{\al_{j_{i}}}}\rangle_s. 
\tionl
\end{proposition}
Using the random measure $M$ from \S\ref{sec:mu.in}, we construct compensated-covariation stable families of martingales: 
For $ \al \in  L^2_1$, we define the square integrable martingale 
    \equa
   Y^\al_t:= \sigma \int_0^t \al(s,0)  dW_s + \int_{[0,t]\times \Rbb_0} \al(s,x)\tilde{\mathcal{N}}(\rmd s,\rmd x), \quad t\in [0,T].
 \tion
 We will call these processes \emph{Engelbert martingales}, they are a generalisation of the processes introduced in \cite[Equation (39)]{DTE16}.
 It holds  $$ [Y^{\al_1}, Y^{\al_2}]_t=  \sigma^2 \int_0^t \al_1(s,0) \al_2(s,0) ds +  \int_{[0, t]\times\Rbb_0} \al_1(s,x)\al_2(s,x)\,\mathcal{N}(\rmd s,\rmd x) $$
and
$$ \langle Y^{\al_1}, Y^{\al_2}\rangle_t= \sigma^2 \int_0^t \al_1(s,0) \al_2(s,0) ds +   \int_{[0, t]\times\Rbb_0} \al_1(s,x)\al_2(s,x)\, \rmd s \nu(\rmd x). $$

For $\Lm=\bigcap_{p\geq2} L^p_1$, it is clear that the family $\Xscr := \{Y^\al: \al \in \Lm \}$ is compensated-covariation stable. 
Note that, for $Y\p{\al_1},\ldots,Y\p{\al_k}\in\Xscr$, $k \ge 2$, we get the following identity 
\begin{equation*} 
Y\p{\al_{1:k}}_t:=[Y\p{\al_{1:(k-1)}},Y\p{\al_{k}}]_t-\pb{Y\p{\al_{1:(k-1)}}}{Y\p{\al_{k}}}_t=\int_{[0,t]\times\Rbb_0} \,\,  \bigg(\prod_{j=1}^k\al_j(s,x)\bigg)   \tilde{\mathcal{N}}(\rmd s,\rmd x).
\end{equation*}

\begin{lemma} The linear hull of the iterated integrals
\equa
 \Big \{ (J_k(\hat f)_t)_{t\in [0,T]}:   f =  {\al_1\otimes...\otimes \al_k}\quad \text{with} \,\, \al_1,...,\al_k \in \bigcap_{p\geq2} L^p_1, \,\, k=0,1,... \Big \}
 \tion
is compensated-covariation stable.
\end{lemma}
\begin{proof}For symmetric kernels $f$ and $g$ we have
\equal  \label{square-brackets}
 [J_n(f), J_k(g)]_T &=&  \int_{(0,T]\times\R} J_{n-1}(f(\dots,(t,x)) )_{t-}\,\, J_{k-1}(g(\dots,(t,x)) )_{t-}  \notag\\
 &&\hspace{15em} \times (\sigma^2 \rmd t \delta_0(\rmd x) + \mathcal{N}(\rmd t,\rmd x))
 \tionl
and 
\equal  \label{angle-brackets}
\langle J_n(f), J_k(g)\rangle_T  & =&  \int_{(0,T]\times\R} J_{n-1}(f(\dots,(t,x)) )_{t-} \, J_{k-1}(g(\dots,(t,x) ))_{t-}  \notag \\
&&\hspace{15em} \times (\sigma^2 \rmd t \delta_0(\rmd x)  + \rmd t d\nu(x))
\tionl
Now we use the product formula for two factors from Lee and Shih \cite[Theorem 3.5]{LeeShih04} (or \cite{DoPecc18} for the pure jump case)  which states that the product of two multiple integrals  of order $n$ and $k$ 
is equal to a linear combination of multiple integrals of order less or equal to $n+k$.   Moreover, from  \cite[equation $(21)$]{LeeShih04} it can be seen that the kernels
of these multiple integrals are built from tensor products of functions from $\bigcap_{p\geq2} L^p_1.$
\end{proof}   


\section{Product formulas}\label{sec:fi.pro.for}
\subsection{Product formula for iterated integrals}\label{subsec:prod.it.int}
The representation formula \eqref{eq:rep.pol} will be used  to  derive a product formula  for $ \prod_{j=1}^N  J_{m_j}( f^{(j)})_T.$   First we consider a special case  assuming that  the kernels are given by tensor products: $f^{(j)}  = \al_j^{\otimes m_j}$  with $\al_j  \in \bigcap_{p\geq2} L^p_1.$  We introduce some notation.
Let 
\equal \label{the s-set}
{\cS} := \{{\bf s}=\{j_1,...,j_i\}: 1\le j_1<...<j_i\le N  \,\, \text{for} \,\, i=1,...,N\}.
\tionl
Then $\cS = 2^{\{1,...,N\}} \setminus \{\emptyset\}$ has $\sn :=2^N-1$ elements  which we denote by
\equal \label{the en-set}
\cS=  \{{\bf e}_1,...,{\bf e}_{\ssn}\}.
\tionl
\begin{proposition}[Product formula for iterated integrals]  \label{product-iterated} Let  $\al_j \in \bigcap_{p\geq2} L^p_1$  for $j=1,...,N$.  Then 
\equal \label{product-to-iterated}
\prod_{j=1}^N J_{m_j}( \al_j^{\otimes m_j}) _T   \!\!&=\!\!\!\!\!& \sum_{ \!\!\! k \le m_1+...+m_N}
\sum_{({\bf s}_1,...,{\bf s}_k)  \in A_k } \,\,
 \int_{\Delta_{\,T}(k)} 
\al_{({\bf s}_1)}(z_1) \cM^{({\bf s}_1)}(dz_1) \ldots       \al_{({\bf s}_k)}(z_k)  \cM^{({\bf s}_k)}( dz_k)  \notag \\
 \tionl 
where $ z_n:= (t_n, x_n),$ and 
$$ \al_{({\bf s}_n)}(z_n) := \prod_{\ell \in {\bf s}_n}  \al_{\ell}(z_n), $$ 
\equal \label{Akset}
A_k:=\left  \{({\bf s}_1,...,   {\bf s}_k) \in \cS^k:  \, \sum_{n=1}^k  \one_{ {\bf s}_n} (j) =m_j,  \,\, \forall j =1,...,N \right \},
\tionl
and $ \Delta_{\,T}(k) $ stands for the 'simplex' defined in \eqref{Delta}.
 For any  ${\bf s}_n \in  \cS$ the integrators are given by (here $|{\bf s}_n| $ denotes the cardinality of the set ${\bf s}_n$)
 \equal \label{Y-explained}
\cM^{({\bf s}_n)}(dz_n) := \left \{ \begin{array}{ll} M(dt_n,dx_n)   & \text{ if }  \, \, |{\bf s}_n| =1,     \\ \\
 \m(dt_n,dx_n)+ \tilde{\mathcal{N}}(dt_n,dx_n)     & \text{ if } \,\,|{\bf s}_n| =2 ,  \\\\
         dt_n\nu(dx_n)     +     \tilde{\mathcal{N}}(dt_n,dx_n)           & \text{ if }\,\,|{\bf s}_n| \ge 3.   \\\\
          \end{array}  \right . 
\tionl
\end{proposition}
The splitting into the cases $|{\bf s}_n| =1, |{\bf s}_n| =2$ and $|{\bf s}_n| \ge 3$ in \eqref{Y-explained} can be explained by \eqref{square-brackets} and \eqref{angle-brackets}. Note that the terms related to the Brownian part vanish for $|{\bf s}_n| \ge 3$ due to the iterated use of  \eqref{iteration-comp-cov}.
\begin{proof}[Proof of Proposition \ref{product-iterated}.]

We  use \eqref{eq:rep.pol} for
$$ X^{\al_j}_T :=J_{m_j}(\al_j^{\otimes m_j})_T
= \int_{(0,T]\times\R} J_{m_{j}-1}(\al_j^{\otimes (m_j-1)} )_{t-} \,\,\al_j(t,x) M(\rmd t,\rmd x) .$$           
Then   
\equal \label{2-compcov}
 X^{\al_{j_{1:2}}}_T  =   
\int_{(0,T]\times\Rbb_0} J_{m_{j_1}-1}(\al_{j_1}^{\otimes (m_{j_1}-1)} )_{t-}   \,
J_{m_{j_2}-1}(\al_{j_2}^{\otimes (m_{j_2}-1)} )_{t-} \,\al_{j_1}(t,x) \al_{j_2}(t,x) \tilde{\mathcal{N}}(\rmd t,\rmd x).  \notag\\\tionl
Further
\equal \label{2-predcov}
 \pb{X^{\al_{j_1}}}{X^{\al_{j_2}}}_T &= &  
\int_{(0,T]\times\R} J_{m_{j_1}-1}(\al_{j_1}^{\otimes (m_{j_1}-1)} )_{t-} J_{m_{j_2}-1}(\al_{j_2}^{\otimes (m_{j_2}-1)} )_{t-} \,\al_{j_1}(t,x) \al_{j_2}(t,x)  \notag  \\ 
&& \hspace{18em} \times \m(\rmd t,\rmd x) ,\tionl
and for more than two processes  we have
\equal \label{eq_with_N-tilde}
X^{\al_{j_{1:i}}}_T  = \int_{(0,T]\times\Rbb_0} \prod_{k=1}^i  \left (J_{m_{j_k}-1}(\al_{j_k}^{\otimes (m_{j_k}-1)} )_{t-} \, \,
 \al_{j_k}(t,x) \right ) \tilde{\mathcal{N}}(\rmd t,\rmd x),
 \tionl         
\equal \label{eq_with_measure} 
\pb{X^{\al_{{j_1}: (i-1)}}}{X^{\al_{j_i}}}_T= \int_{(0,T]\times\Rbb_0} \prod_{k=1}^i \left  (J_{m_{j_k}-1}(\al_{j_k}^{\otimes (m_{j_k}-1)} )_{t-}  \,\,
 \al_{j_k}(t,x) \right ) \rmd t \nu(\rmd x).  
 \tionl   
Then  \eqref{eq:rep.pol} implies that
\equal \label{first-application}
&& \prod_{j=1}^N J_{m_j}( \al_j^{\otimes m_j})_T  \\
&=\!& \sum_{i=1}^{2^N-1}  \int_{(0,T]\times\R}  \Big (\prod_{\begin{subarray}{c}n=1\\n\notin {\bf e}_i \end{subarray}}^N  J_{m_n}(\al_{n}^{\otimes m_n} )_{t-}  \Big)  \Big ( \prod_{j_k\in {\bf e}_i } J_{m_{j_k}-1}(\al_{j_k}^{\otimes (m_{j_k}-1)} )_{t-}  \Big) 
 \al_{({\bf e}_i )}(t,x) \cM^{({\bf e}_i)}(\rmd t, \rmd x).   \notag
 \tionl
Indeed, if   $|{\bf e}_i|=1$  only the first sum in \eqref{eq:rep.pol}
contributes, and   we have   that $$\cM^{({\bf e}_i)} =M.$$   If  ${\bf e}_i=\{ i_1,i_2\},$ then    $|{\bf e}_i|=2,$ and  we have two summands in \eqref{eq:rep.pol} with integrators 
$X^{\al_{i_{1:2}}}$  and  $\langle  X^{\al_{j_1}} , X^{\al_{j_2}} \rangle,$   so that from \eqref{2-compcov} and  \eqref{2-predcov} we conclude that
 $$\cM^{({\bf e}_i)} (\rmd t,\rmd x) = \m(\rmd t,\rmd x)+ \tilde{\mathcal{N}}(\rmd t,\rmd x) .$$
For the remaining cases $|{\bf e}_i|\ge 3$ we have two summands in \eqref{eq:rep.pol}, and we get from  \eqref{eq_with_N-tilde} and  \eqref{eq_with_measure} 
that 
$$   \cM^{({\bf e}_i)} (\rmd t,\rmd x)   = \rmd t\nu(\rmd x)     +    \tilde{\mathcal{N}}(\rmd t,\rmd x).$$
This proves \eqref{Y-explained}.

Now we apply the relation \eqref{first-application}   to $ \Big (\prod_{\begin{subarray}{c}n=1\\n\notin {\bf e}_i \end{subarray}}^N  J_{m_n}(\al_{n}^{\otimes m_n} )_{t-}  \Big)  \Big ( \prod_{j_k\in {\bf e}_i } J_{m_{j_k}-1}(\al_{j_k}^{\otimes (m_{j_k}-1)} )_{t-}  \Big)$ and repeat this procedure until no products of  iterated  integrals are left as integrands.
The resulting integrals are of the form
$$\int_{\Delta_{\,T}(k)} 
\al_{({\bf s}_1)}(z_1) \cM^{({\bf s}_1)}(dz_1) \ldots       \al_{({\bf s}_{k})}(z_k)  \cM^{({\bf s}_k)}( dz_k), $$
where $\max\{m_1,...,m_N\} \le k \le m_1+...+m_N,$ and  only those  sequences  $({\bf s}_1,...,   {\bf s}_k) \in   \cS^k$  appear which
satisfy
$$ \sum_{n=1}^k  \one_{ {\bf s}_n} (j) =m_j$$
for all $j =1,...,N.$ This condition describes how the iterated  integrals  are  used in this procedure and their order is diminished to zero  to get the $\al_{({\bf s}_{n})}(z_n).$
This leads to $$A_k:= \left \{({\bf s}_1,...,   {\bf s}_k) \in \cS^k:  \, \sum_{n=1}^k  \one_{ {\bf s}_n} (j) =m_j,  \,\, \forall  j =1,...,N    \right \},$$
which consists of those tuples  $({\bf s}_1,...,   {\bf s}_k)$ such that each number $j$ (between $1$ and $N$) is contained in exactly  $m_j$  of the  ${\bf s}_n.$
In this way we get  relation \eqref{product-to-iterated}. 
 \end{proof} \smallskip
  We have achieved now a product formula for {\it iterated} integrals  where the  kernels are tensor products with strong integrability conditions. In the next section we  derive a product formula for {\it multiple} integrals w.r.t.~$M.$
	
\subsection{Product formula for multiple integrals}\label{susec:prod.for.mul.int} 
Our aim is to rewrite the product formula for iterated integrals obtained in Proposition \ref{product-iterated} in terms of multiple integrals w.r.t.~$M(\rmd t,\rmd x)$.  For this  we  need some preparation.

\paragraph{Preparation.}
 
We  rewrite the  integrals w.r.t. $ \cM^{({\bf s}_k)}$ as sums of integrals where the integrators are either $M$ or $\m.$ We replace 
$\al_{({\bf s}_n)}(t_n, x_n) \tilde{\mathcal{N}}(dt_n,dx_n)$ by  $\al_{({\bf s}_n)}(t_n, x_n) \one_{\R_0}(x_n) M(dt_n,dx_n).$ By \eqref{Y-explained} one has

 \equal \label{random-deterministic}
 && \int_{\Delta_{\,T}(k)} 
\al_{({\bf s}_1)}(z_1) \cM^{({\bf s}_1)}(dz_1) \ldots       \al_{({\bf s}_k)}(z_k)  \cM^{({\bf s}_k)}( dz_k)   \notag\\
&&=  \sum_{{\bf i} \in \{0,1\}^{k}} \int_{\Delta_{\,T}(k)}  \al_{{\bf s}_1,i_1}(z_1)\cM^{ i_1}(dz_1) \ldots       
\al_{{\bf s}_{k},i_k}(z_k)  \cM^{i_k}( dz_k),
\tionl 
where we use $i_n=1$ for the random part and $i_n=0$ for the deterministic part:
$$   \cM^{i_n}(\rmd t, \rmd x) := \left \{ \begin{array}{ll} M(\rmd t,\rmd x)   & \text{ if }  \, \,    i_n=1  \\ \\
\m(\rmd t,\rmd x)  & \text{ if } \,\,  i_n=0.\\
        \end{array}      \right . $$
 Recall that  ${\bf s}_1,...,{\bf s}_{k} \in \cS.$ For the  ${\bf e}_n \in \cS$  (see \eqref{the s-set} and \eqref{the en-set}) we put 
\equal  \label{many-choices}  
\al_{{\bf e}_n, i_n} (z_n)  := \left \{ \begin{array}{ll} \prod_{\ell \in {\bf e}_n}  \al_{\ell} (z_n)  & \text{ if }  \, \, |{\bf e}_n| =1,   i_n=1  \\ \\
\prod_{\ell \in {\bf e}_n}  \al_{\ell} (z_n) \one_{[0,T]\times\R_0}(z_n)  & \text{ if }  \, \, |{\bf e}_n| \ge 2,   i_n=1  \\ \\
0 & \text{ if }  \, \, |{\bf e}_n| =1,   i_n=0  \\ \\
\prod_{\ell \in {\bf e}_n}  \al_{\ell} (z_n)    & \text{ if } \,\,|{\bf e}_n| =2 , i_n=0 \\\\
\prod_{\ell \in {\bf e}_n}  \al_{\ell} (z_n)   \one_{[0,T]\times\R_0}(z_n)    & \text{ if } \,\,|{\bf e}_n| \ge 3 , i_n=0. \\
  \end{array}      \right . \tionl
So we have  integrands of the form

$$ \prod_{n=1}^{k} \al_{{\bf s}_n, i_n} (t_n,x_n) \one_{\{0\le t_1\le ...\le t_k \le T\}}.  $$

\begin{definition} \label{theD-N-set} We set (recall that $\sn=2^N -1$)
  \equal \label{Dsn}
&&D_N:= \{   (l,l^o)=((l_1,...,l_{\ssn}),(l^o_1,...,l^o_{\ssn})):  l_k, l^o_k \in \{0,1,...\} \,\, \text{for} \,\, k=1,...,\sn  \notag\\
&&\hspace{9em} \text {such that} \,\,  \sum_{n=1}^{\ssn} (l_n+ l_n^o) \one_{{\bf e}_n}(j) =m_j \,\,\text{for} \,\,\, j=1,...,N  \}.    
\tionl
Moreover, for $l=(l_1,...,l_{\ssn})$  and   $ l^o =  (l^o_1,...,l^o_{\ssn})$ we set
\begin{enumerate}[(i)]
\item  
$ |l| := \sum_{n=1}^{\ssn} l_n, \quad   |l^o| := \sum_{n=1}^{\ssn} l^o_n, \quad  l! := l_1!\cdots l_{\ssn}!,  \quad   \text{and }  \quad    l^o! := l^o_1!\cdots l^o_{\ssn}!,$
\item and for  $ \al_{{\bf e}_n,i_n}$ defined in \eqref{many-choices} we put 
$$ \al^{\otimes l_n}_{{\bf e}_n,1}:=\underbrace{  \al_{{\bf e}_n,1} \otimes...\otimes  \al_{{\bf e}_n,1} }_{l_n \, \text{times}}  \quad \text{and} \quad  \al^{\otimes l^o_n}_{{\bf e}_n,0}:=\underbrace{  \al_{{\bf e}_n,0} \otimes...\otimes  \al_{{\bf e}_n,0} }_{l_n^o \, \text{times}},$$
\equal \label{the-alpha-l}
  \al^{\otimes l} := \al_{{\bf e}_1,1}^{\otimes l_1} \otimes ... \otimes\al_{{\bf e}_{\ssn},1}^{\otimes l_{\ssn}}
  \quad \text{and }  \quad  \al^{ \otimes l^o} := \al_{{\bf e}_1,0}^{\otimes l^o_1} \otimes ... \otimes \al_{{\bf e}_{\ssn},0}^{\otimes l^o_{\ssn}}.
  \tionl
\end{enumerate}
\end{definition}
The set  $D_N$ has the same function as the set $A_k$ in \eqref{Akset}. We use the superscript $^o$ as reminder that $l^o$ is the vector related to the deterministic case, i.e.~to the integration w.r.t.~$\m.$ 
 \smallskip\\
We are now in the position to state a first product formula for multiple integrals which is a consequence of Proposition \ref{product-iterated}. The proof is postponed to the appendix.

\begin{proposition} \label{alpha-products}
Let  $\al_j \in \bigcap_{p\geq2} L^p_1$  for $j=1,..,N.$ Then 
\equal  \label{first-product-formula}
&&\!\!\prod_{j=1}^N I_{m_j}( \al_j^{\otimes m_j})_T  
= \!\!\!\sum_{k \le m_1+\dots+ m_N}  \sum_{\begin{subarray}{c}|l|=k,\\ (l,l^o) \in D_N \end{subarray}}   \frac{m_1!\cdots m_N!}{l!\, l^o !} 
\left (\int_{((0,T]\times \R)^{|l^o|} }   \al^{ \otimes l^o}d\m^{\otimes |l^o|} \right )\,\,  I_k (\al^{\otimes l})_T, \notag \\
\tionl
where the notation from Definition \ref{theD-N-set} is used.
\end{proposition}

\smallskip

We proceed with an example  to illustrate the above notation.

\begin{example} \label{three-functions} Let us assume that the L\'evy process does only have the jump part, i.e. $M(\rmd t,\rmd x) = \tilde \cN(\rmd t,\rmd x),$ so that we do not need to multiply with  $\one_{[0,T]\times\R_0}(z_n)$ in
\eqref{many-choices} which is used to separate the  contribution coming from the Gaussian part. Assume that $\al_1,\al_2,\al_3 \in \bigcap_{p\geq2} L^p_1$  and put
$$ f^{(1)}(z_1,...,z_{m_1}) = \al_1^{\otimes m_1} ( z_1,...,z_{m_1})$$
$$ f^{(2)}(z_1,...,z_{m_2}) = \al_2^{\otimes m_2} ( z_1,...,z_{m_2})$$
$$ f^{(3)}(z_1,...,z_{m_3}) = \al_3^{\otimes m_3} ( z_1,...,z_{m_3}).$$
Then  the set $\cS$  given in \eqref{the s-set}  contains $2^{3}-1$ elements. For example (the order is not relevant) we put 
$$ {\bf e}_1 = \{1\}, \,\,\,   {\bf e}_2 =\{2\}, \,\,\,  {\bf e}_3 =\{3\}, \,\,\,  {\bf e}_4  =\{1,2\}, \,\,\,{\bf e}_5 =\{1,3\}, \,\,\, {\bf e}_6 =\{2,3\}, \,\,\, {\bf e}_7  =\{1,2,3\}.$$
For  $0 \le k \le m_1+m_2+m_3$ we require that $|l| =l_1 +...+l_7 =k$ and from \eqref{the-alpha-l} we get
$$ \al^{\otimes l} =  \al_1^{\otimes l_1}\otimes  \al_2^{\otimes l_2}\otimes \al_3^{\otimes l_3}\otimes  (\al_1\al_2)^{\otimes l_4} \otimes(\al_1\al_3)^{\otimes l_5} \otimes(\al_2 \al_3)^{\otimes l_6} \otimes(\al_1\al_2\al_3)^{\otimes l_7}  $$
$$ \al^{\otimes l^o} =  \al_1^{\otimes l^o_1}\otimes  \al_2^{\otimes l^o_2}\otimes \al_3^{\otimes l^o_3}\otimes  (\al_1\al_2)^{\otimes l^o_4} \otimes(\al_1\al_3)^{\otimes l^o_5} \otimes(\al_2 \al_3)^{\otimes l^o_6} \otimes(\al_1\al_2\al_3)^{\otimes l^o_7}  $$
Only  pairs $(l,l^o)$ are used where  the number of functions {\new $\al_1 $}    in $ \al^{\otimes l} $ and $ \al^{\otimes l^o}$ altogether equals  $m_1,$
and likewise the number of functions $\al_2 $ and $\al_3$  in $ \al^{\otimes l} $ and $ \al^{\otimes l^o}$ equals to $m_2$ and $m_3,$ respectively. This we ensure by
using the set $D_N.$
\end{example} 

In  Proposition \ref{alpha-products} we stated a product formula for multiple integrals for the special case where the integrands are given by
$f^{(j)} = \al_j^{\otimes m_j}.$
In order to work with arbitrary $f^{(j)}  \in \bigcap_{p\geq2} L^p_{m_j}$  we need to replace our notation  $\al^{ \otimes l^o}$ and  $\al^{\otimes l}$  such that it works for the general case.
We adapt the notation used in \cite{DoPecc18}: \bigskip

\begin{definition}[contractions and identifications] \label{contractions_and_identifications}
Let  $f^{(j)}  \in \bigcap_{p\geq2} L^p_{m_j}$  be symmetric  functions for  $j=1,...,N.$

\begin{enumerate}[(a)]
\item     For $(l,l^o) \in D_N$ and  ${\bf e}_n \in \cS$  given in \eqref{the s-set}  we define  for  $n=1,..., \sn$ {\it identifications} of $l_n$  variables and   {\it  contractions}  of $l^o_n$ variables:    
\begin{align*}
 &\bigstar^l_{l^o}  (f^{(1)},...,f^{(N)} )\left((z_{1:l_n}^{{\bf e}_n})_{1\leq n\leq s_N}\right) \\
 &\quad  :=
 \int_{({(0,T]}\times\R)^{|l^o|}} \prod_{j=1}^Nf^{(j)}\left((z_{1:l_n}^{{\bf e}_n})_{1\leq n\leq s_N, j\in {\bf e}_n},(z_{1:l_n^o}^{{\bf e}_n,0})_{1\leq n\leq s_N, j\in {\bf e}_n}\right)\\
&\quad\quad\quad\quad\quad\quad\quad\quad  \times \prod_{1\leq n\leq s_N}c_{{\bf e}_n}\left(z_{1:l_n}^{{\bf e}_n},z_{1:l_n^o}^{{\bf e}_n,0}\right)
 d\m^{\otimes |l^o|}\big((z_{1:l_n^o}^{{\bf e}_n,0})_{1\leq n\leq s_N}\big),
\end{align*}
constituting a function in the $|l|$ variables $(z_{1:l_n}^{{\bf e}_n})_{1\leq n\leq s_N},$
 and
   \equal\label{cen}  c_{{\bf e}_n}(z_{1:l_n}^{{\bf e}_n},z_{1:l_n^o}^{{\bf e}_n,0}) := \left \{ \begin{array}{ll}1  & \text{ if }  \, \, |{\bf e}_n| =1,   l_n^o=0  \\ \\
0  & \text{ if }  \, \, |{\bf e}_n| = 1,   l_n^o\ge 1  \\ \\
\prod_{i=1}^{l_n}  \one_{(0,T]\times\R_0}(z^{{\bf e}_n}_i)      & \text{ if } \,\,|{\bf e}_n| =2  \\\\
\prod_{i=1}^{l_n^o}  \one_{(0,T]\times\R_0}(z^{{\bf e}_n,0}_i)    & \text{ if } \,\,|{\bf e}_n| \ge 3, \\
  \end{array}      \right .\tionl  
where $\prod_{i=1}^0 \cdots  :=1.$ The variables from the kernels  $f^{(1)},...,f^{(N)}$   are grouped into tuples $z_{1:l_n}^{{\bf e}_n}= (z^{{\bf e}_n}_1,..., z^{{\bf e}_n}_{l_n}) $  and $z_{1:l_n^o}^{{\bf e}_n,0}=(z^{{\bf e}_n,0}_1,..., z^{{\bf e}_n,0}_{l^o_n})$ with the convention that  there is no  tuple  $z_{1:l_n}^{{\bf e}_n}$ if $l_n=0,$ and likewise for $l_n^o=0$.
 The notation $(z_{1:l_n}^{{\bf e}_n})_{1\leq n\leq s_N, j\in {\bf e}_n}$ for the variables in  $f^{(j)}$
 means that we  only use the tuples of variables  $z_{1:l_n}^{{\bf e}_n}$ for which it holds that $j\in {\bf e}_n,$ and the same holds for $(z_{1:l_n^o}^{{\bf e}_n,0})_{1\leq n\leq s_N, j\in {\bf e}_n}.$

\item The symmetrization of $\bigstar^l_{l^o}  (f^{(1)},...,f^{(N)})$ is denoted by $\widehat \bigstar^l_{l^o}  (f^{(1)},...,f^{(N)}).$
\end{enumerate}
\end{definition}

\begin{example}
\begin{enumerate}[(a)] \item
Note  that  the expression $\bigstar^l_{l^o}  (f^{(1)},...,f^{(N)})$ is in general not symmetric: Take for example $f^{(1)}(z_1,z_2) = \al_1(z_1)\al_1(z_2)$ and 
$f^{(2)}(z_1)=\al_2(z_1),$ and  ${\bf e}_1 =\{1\}, {\bf e}_2 =\{2\}, {\bf e}_3 =\{1,2\}.$ Choose $l=(1,0,1)$ and $l^o=(0,0,0).$   To see that  $(l,l^o) \in  D_N$ we check 
$l_1+l_3 =2= m_1$ and $l_3=1 =m_2.$ Then  
$$\bigstar^l_{l^o}  (f^{(1)},f^{(2)} ) (z^{{\bf e}_1}_1, z^{{\bf e}_3}_1 ) = \al_1(z^{{\bf e}_1}_1)\al_1(z^{{\bf e}_3}_1) \al_2(z^{{\bf e}_3}_1) \one_{[0,T]\times\R_0}(z^{{\bf e}_3}_1) .$$
Symmetry would require the relation $\al_2(z^{{\bf e}_3}_1)  \one_{(0,T]\times\R_0}(z^{{\bf e}_3}_1)  =\al_2(z^{{\bf e}_1}_1)\one_{(0,T]\times\R_0}(z^{{\bf e}_1}_1)$ for all $z^{{\bf e}_1}_1, z^{{\bf e}_3}_1 \in (0,T]\times\R.$

\item
\bigskip
Now we  take $N=3$, $m_1=13, m_2=10, m_3=8$. Then $ |\cS|= \sn = 2^N-1=7$  and 
\begin{align*}
\{{\bf e}_1,\dotsc,{\bf e}_{7}\}=\left\{ \{1\},\{2\},\{3\}, \{1,2\}, \{1, 3\}, \{2, 3\}, \{1,2,3\}\right\}.
\end{align*}
Assume that $(l,l^o)=((5,2,1,1,1,1,1),(0, 0,0,2,1,1,2))$.  We have that   $(l,l^o)\in  D_N.$  By the above definition,
\begin{align*}
& \hspace*{-1.7em} \bigstar_{l^o}^l(f^{(1)},f^{(2)},f^{(3)})\big(z^{\{1\}}_1,\dotsc,z^{\{1\}}_5,z^{\{2\}}_1, z^{\{2\}}_2, z^{\{3\}}_1, z^{\{1,2\}}_1,z^{\{1,3\}}_1, z^{\{2,3\}}_1, z^{\{1,2,3\}}_1\big)\\
&\hspace*{-1.7em}  =\int_{({{(0,T]}\times\R_0})^{6}} \hspace*{-1.1em}  f^{(1)}\big(z^{\{1\}}_1,\dotsc,z^{\{1\}}_5, z^{\{1,2\}}_1, z^{\{1,3\}}_1, z^{\{1,2,3\}}_1, z^{\{1,2\},0}_1, z^{\{1,2\},0}_2, z^{\{1,3\},0}_1, z^{\{1,2,3\},0}_1, z^{\{1,2,3\},0}_2\big)\\
& \quad\quad\quad \times f^{(2)}\big(z^{\{2\}}_1, z^{\{2\}}_2, z^{\{1,2\}}_1, z^{\{2,3\}}_1, z^{\{1,2,3\}}_1, z^{\{1,2\},0}_1, z^{\{1,2\},0}_2, z^{\{2,3\},0}_1, z^{\{1,2,3\},0}_1, z^{\{1,2,3\},0}_2\big) \\ \\
&\quad\quad\quad  \times f^{(3)}\big(z^{\{3\}}_1, z^{\{1,3\}}_1, z^{\{2,3\}}_1, z^{\{1,2,3\}}_1, z^{\{1,3\},0}_1, z^{\{2,3\},0}_1, z^{\{1,2,3\},0}_1, z^{\{1,2,3\},0}_2\big)
\\ \\
&\quad\quad\quad \times \one_{[0,T]\times\R_0}\big(z^{\{1,2\}}_1\big)\one_{[0,T]\times\R_0}\big(z^{\{1,3\}}_1\big)\one_{[0,T]\times\R_0}\big(z^{\{2,3\}}_1\big)\one_{[0,T]\times\R_0}\big(z^{\{1,2,3\},0}_1\big) \\ \\
&\quad\quad\quad \times \one_{[0,T]\times\R_0}\big(z^{\{1,2,3\},0}_2\big)
 d\m^{\otimes 6}\big(z^{\{1,2\},0}_1, z^{\{1,2\},0}_2,z^{\{1,3\},0}_1,z^{\{2,3\},0}_1, z^{\{1,2,3\},0}_1, z^{\{1,2,3\},0}_2\big).
\end{align*}

Indeed,  by integrating over $|l^o|=6$ variables, we get a function in $|l|=12$ variables.  Note that the components of $l^o$ are $0$ for all singletons $\{j\}, j=1,2,3$, because otherwise the corresponding term $c_{\{j\}}$ would yield zero. 
\end{enumerate}
\end{example}
 
\begin{rem} \label{en}  It is of interest to have the product formula for multiple integrals for the special cases $M(\rmd t, \rmd x) = \sigma dW_t \delta_0(\rmd x)$  and  $M(\rmd t, \rmd x) =  \tilde \cN(\rmd t,\rmd x).$ They can be derived easily from the general expression by adjusting $c_{{\bf e}_n}$:
\begin{enumerate}[(a)] 
\item \label{jump}  For the jump case $M(\rmd t, \rmd x) =  \tilde \cN(\rmd t,\rmd x)$  we use  
\equa  c_{{\bf e}_n}(z_{1:l_n}^{{\bf e}_n},z_{1:l_n^o}^{{\bf e}_n,0}) := \left \{ \begin{array}{ll}0  & \text{ if }  \, \, |{\bf e}_n| = 1,   l^o_n \ge 1  \\ \\
1  & \text{ otherwise}.    \\ 
  \end{array}      \right .\tion
\item\label{cbrown} For the Brownian motion case  $M(\rmd t, \rmd x) = \sigma dW_t \delta_0(\rmd x)$ we use 
\equal\label{e:cBrown}  c_{{\bf e}_n}(z_{1:l_n}^{{\bf e}_n},z_{1:l_n^o}^{{\bf e}_n,0}) := \left \{ \begin{array}{ll}1  & \text{ if }  \, \, |{\bf e}_n| =1,   l_n^o=0  \\ \\
1    & \text{ if } \,\,|{\bf e}_n| =2,   l_n^o\ge 1, l_n=0   \\\\
0   & \text{otherwise}. \\
  \end{array}      \right .\tionl
\end{enumerate}
\end{rem}

In  \cite[Theorem 2.2]{Ag20}  a product formula is stated for multiple integrals w.r.t.~the compensated Poisson random measure. The expression coincides  with 
\eqref{product-formula} below if one uses Remark \ref{en}\eqref{jump}.  The main result of this paper is Theorem \ref{L2condition-thm} where we establish sufficient 
 integrability conditions on the kernels for the product formula. Its proof uses similar ideas as in \cite{DTG19} for products of iterated integrals.

\begin{theorem}  \label{L2condition-thm}
Assume symmetric  $f^{(j)} \in L^2_{m_j}$  for $j=1,...,N.$ If for all $(l,l^o) \in D_N$ with $|l|=k$  ($k \le m_1+...+m_N$) we have 
\equal \label{L2-condition} 
&& \int_{((0,T]\times \R)^k} \left (   \widehat \bigstar^l_{l^o}  (|f^{(1)}|,...,|f^{(N)} |)       \right )^2 d\m^{\otimes k} < \infty , 
\tionl
then
$$ \prod_{j=1}^N I_{m_j}(f^{(j)})_T \quad \text{  is square integrable}$$
 and 
 \equal \label{product-formula}
\prod_{j=1}^N I_{m_j}(f^{(j)})_T   &=&  \sum_{k \le m_1+\dots+ m_N}  I_k \left ( \sum_{\begin{subarray}{c}|l|=k,\\ (l,l^o) \in D_N \end{subarray}} \!  \frac{m_1!\cdots m_N!}{l!\, l^o !} \,\,
 \widehat \bigstar^l_{l^o}  (f^{(1)},...,f^{(N)} ) \right )_T .
\tionl
In particular,
 \equal \label{expectation-product-formula}
\EE \left ( \prod_{j=1}^N I_{m_j}(f^{(j)})_T \right )  &=&  \sum_{\begin{subarray}{c} ((0,...,0),l^o) \in D_N \end{subarray}} \!  \frac{m_1!\cdots m_N!}{l^o !} \,\,
  \bigstar^{(0,...,0)}_{l^o}  (f^{(1)},...,f^{(N)} ) .
\tionl

\end{theorem}
\begin{proof} 
We will use Proposition \ref{alpha-products}. 
By the linearity  of the multiple integral  we may move  in \eqref{first-product-formula} the sum over $|l|=k, (l,l^o) \in D_N$ inside the multiple integral $I_k$ so that we have a sum of {\it orthogonal} multiple integrals. 
To derive  a representation for integrands  $f^{(j)}=\otimes_{i=1}^{m_j} \al_i$ (i.e. with different functions  $\al_i \in \bigcap_{p\geq2} L^p_1$ as factors)  from \eqref{first-product-formula} we will  use  the polarization identity like in \cite{LeeShih04} (see the comment on page 85 above Theorem 3.5). 
For this we need  that \eqref{first-product-formula} is  a multilinear symmetric form. The multilinearity is clear.  To realize that it is also symmetric,  i.e.~invariant under all permutations of the variables, note that we can replace the integrand $\alpha^{\otimes l^o}$ by its symmetrization $\alpha^{\hat{\otimes} l^o}$ since
$$\int_{((0,T]\times \R)^{|l^o|} }   \al^{ \otimes l^o}d\m^{\otimes |l^o|}  =  \int_{((0,T]\times \R)^{|l^o|} }   \al^{ \hat \otimes l^o}d\m^{\otimes |l^o|}. $$
Moreover, by the properties of the multiple integrals we have
$$  I_k (\al^{\otimes l})_T = I_k (\al^{\hat \otimes l})_T .$$
For the l.h.s.~of \eqref{first-product-formula} we use Lemma  \ref{iterated-and-multiple}  \eqref{with-hat-or-not}   and arrive at \eqref{product-formula} for  $ \prod_{j=1}^N I_{m_j}(f^{(j)}  )$. Indeed, it holds for $f^{(j)} = \al_j^{\otimes  m_j}$ for $j=1,...,N$ that 
$$    \int_{((0,T]\times \R)^{|l^o|}}  \al^{\hat \otimes l} \hat \otimes  \al^{\hat \otimes l^o} d\m^{\otimes  |l^o|}=  \widehat \bigstar^l_{l^o}  (f^{(1)},...,f^{(N)}).$$ 
This follows immediately from Definition  \ref{contractions_and_identifications}  and \eqref{many-choices} if we write $  \al^{ \otimes l} \otimes \al^{\otimes l^o}$ as a product (with the convention that $\prod_{r=1}^0 \dots :=1$) : 
\equa
&& \al^{ \otimes l} \otimes  \al^{\otimes l^o} \left((z_{1:l_n}^{{\bf e}_n})_{1\leq n\leq \ssn},(z_{1:l_n^o}^{{\bf e}_n,0})_{1\leq n\leq \ssn} \right)  \\
& =& \left ( \prod_{r=1}^{l_1} \al_{{\bf e}_1, 1} (z^{{\bf e}_1}_r) \right ) \cdots  \left ( \prod_{r=1}^{l_{\ssn}} \al_{{\bf e}_{\ssn}, 1} (z^{{\bf e}_{\ssn}}_r) \right )     \left ( \prod_{r=1}^{l^o_1} \al_{{\bf e}_1, 0} (z^{{\bf e}_1,0}_r) \right ) \cdots  \left ( \prod_{r=1}^{l^o_{\ssn}} \al_{{\bf e}_{\ssn}, 0} (z^{{\bf e}_{\ssn},0}_r) \right )
\\
&=&
\prod_{j=1}^Nf^{(j)}\left((z_{1:l_n}^{{\bf e}_n})_{1\leq n\leq s_N, j\in {\bf e}_n},(z_{1:l_n^o}^{{\bf e}_n,0})_{1\leq n\leq s_N, j\in {\bf e}_n}\right) \prod_{1\leq n\leq s_N}c_{{\bf e}_n}\left(z_{1:l_n}^{{\bf e}_n},z_{1:l_n^o}^{{\bf e}_n,0}\right).
\tion
\medskip
Now we continue with an approximation:  \smallskip \\
\underline{Step 1.}  The first observation is that  the assertion of Proposition \ref{alpha-products}  extends  from  $f^{(j)}=\otimes_{i=1}^{m_j} \al_i$ to indicator functions $f^{(j)} =\one_{A_j}\in L^2_{m_i}:$ 
 Similar to  \cite[Theorem 2.40]{Foll}  or \cite[Theorem A.8]{GY} 
 one shows that given  an $\varepsilon>0$ and a Borel set $A_j \in  \cB(((0,T]\times \R)^{m_j}) $,  with $ \m^{\otimes m_j}(A_j)<\infty$ there exists an $M \ge 1$ and  disjoint  rectangles
$R_{j,1},\ldots R_{j,M} \subseteq ((0,T]\times \R)^{m_j}$ whose sides are intervals such that
$$ \int_{((0,T]\times \R)^{m_j}} \Big| \one_{A_j} -\sum_{i=1}^M \one_{R_{j,i}} \Big | d\m^{\otimes m_j} = \m^{\otimes m_j}\Big(A_j \Delta \bigcup_{i=1}^M R_{j,i}\Big)   < \varepsilon.$$
The result in \cite[Theorem 2.40]{Foll}   is proven for the Lebesgue measure. But following the proof it is easy to see that it holds for $\m^{\otimes m_j}$ as well. To prove this approximation one can 
also use \cite[Theorem A.8]{GY} together with the fact, that one can write a simple function where the sets are  'half open rectangles'  always as a simple function using disjoint half open rectangles.  \\
We note that $\sum_{i=1}^M \one_{R_{j,i}} \in H_{m_j},$ where $H_{m_j}$  denotes  the linear hull   of the  set 
$$\left \{ \otimes_{i=1}^{m_j} \al_i:   \al_i \in  \cap_{p\geq2} L^p_1  \right  \}. $$
Moreover,  $0\le \sum_{i=1}^M \one_{R_{j,i}} \le 1$.  So we have shown that there exists a sequence $(f_n^{(j)})_{n=1}^\infty \subseteq H_{m_j}$ with 
\equal \label{step-indicator}
|f^{(j)}_n|\le 1  \quad   \text{and} \quad   f^{(j)}_n \rightarrow  \one_{A_j} \quad  \text{ in }  \quad L^2_{m_j},\ 
\text{ as }\ n\to \infty,\quad \ \text{ for all }\ j=1 \ldots,N.
 \tionl
For each $n$ we have that  $\prod_{j=1}^N I_{m_j}(f^{(j)}_n)_T \in L^2(\Pbb)$ and satisfies \eqref{product-formula}. 
To show that  the product $\prod_{j=1}^N I_{m_j}(f^{(j)}_n)_T$ is a Cauchy sequence in $L^2(\Pbb)$  for  $n\to \infty$ we estimate (abbreviating $\overline m_N:= m_1+\dots+ m_N$)
\begin{align} \label{estimate-of-product}
& \left \|\prod_{j=1}^N I_{m_j}(f^{(j)}_n)_T-\prod_{j=1}^N I_{m_j}(f^{(j)}_m)_T \right  \|_{L^2(\Pbb)} \notag \\
&\le   \sum_{r=1}^N \left \| \prod_{j=1}^{r-1} I_{m_j}(f^{(j)}_n)_T I_{m_r}(f^{(r)}_n-  f^{(r)})_T \prod_{j=r+1}^N I_{m_j}(f^{(j)}_m)_T   \right \|_{L^2(\Pbb)}  \notag  \\
&\le    \sum_{r=1}^N
\sum_{k \le \overline m_N}  \sum_{\begin{subarray}{c}|l|=k,\\ (l,l^o) \in D_N \end{subarray}}   \frac{m_1!\cdots m_N!}{l!\, l^o !} 
 \left \|  I_k (\widehat \bigstar^l_{l^o}  (f^{(1)}_n,...,f^{(r-1)}_n, f^{(r)}_n- f^{(r)}_m, f^{(r+1)}_m,   ...,f^{(N)}_m )_T \right \|_{L^2(\Pbb)}    \notag \\
&=  \sum_{r=1}^N    \sum_{k \le \overline m_N}  \sum_{\begin{subarray}{c}|l|=k,\\ (l,l^o) \in D_N \end{subarray}}   \frac{m_1!\cdots m_N!}{l!\, l^o !} k!
   \Big \|  \widehat \bigstar^l_{l^o}  (f^{(1)}_n,...,f^{(r-1)}_n, f^{(r)}_n- f^{(r)}_m, f^{(r+1)}_m,   ...,f^{(N)}_m )    \Big \|_{L^2_k }  \notag \\
   &\le    \sum_{r=1}^N    \sum_{k \le \overline m_N}  \sum_{\begin{subarray}{c}|l|=k,\\ (l,l^o) \in D_N \end{subarray}}   \frac{m_1!\cdots m_N!}{l!\, l^o !} 
  k! \Big \|  \widehat \bigstar^l_{l^o}  (|f^{(1)}_n|,...,|f^{(r-1)}_n|,|f^{(r)}_n- f^{(r)}_m|,|f^{(r+1)}_m|,   ...,|f^{(N)}_m| )    \Big \|_{L^2_k } \notag \\
\end{align}
which converges to zero as $n,m \to \infty$ by dominated convergence since the integrands  are bounded by $ \widehat \bigstar^l_{l^o}  ( \one_{A_1},..., \one_{A_N}) $ and converge in measure to zero. Hence $\prod_{j=1}^N I_{m_j}(f^{(j)}_n)_T $ is a Cauchy sequence in $L^2(\Pbb).$ 
On the other hand, \eqref{step-indicator} implies that  $I_{m_j}(f^{(j)}_n)_T\to I_{m_j}(\one_{A_j})_T$ in $L^2(\Pbb)$ for $j=1,...,N,$
and consequently $\prod_{j=1}^N I_{m_j}(f^{(j)}_n)_T \to \prod_{j=1}^N I_{m_j}(\one_{A_j})_T$ in probability. This gives   
$\prod_{j=1}^N I_{m_j}(\one_{A_j})_T \in L^2(\Pbb),$ as  it is  also the limit of the 
Cauchy sequence in   $ L^2(\Pbb).$
It remains to show that the r.h.s. of  \eqref{product-formula}  (used now for the $\hat f^{(1)}_n,...,\hat f^{(N)}_n$) converges in  $L^2(\Pbb)$   if \eqref{step-indicator} holds. This follows from the convergence  
$$  \widehat \bigstar^l_{l^o}  (\hat f^{(1)}_n,...,\hat f^{(N)}_n)  - \widehat \bigstar^l_{l^o}  (\hat\one_{A_1},...,\hat\one_{A_N})$$
in $L^2_k.$ \smallskip \\
 \underline{Step 2.}  Let  $f^{(i)} \in L^2_{m_i}$ for $i=1,...,N$ satisfy 
\eqref{L2-condition}. Especially, then the r.h.s. of  \eqref{product-formula} is well-defined and in $L^2(\Pbb).$  
We show  \eqref{product-formula} by approximation as follows. Since \eqref{product-formula} holds for indicator functions, it holds for simple functions. 
We  approximate any $f^{(j)}$  by  simple functions $(g^{(j)}_n)_{n=1}^\infty$ such that   $0 \le g^{(j), \pm}_n \uparrow f^{(j),\pm}. $
Applying now \eqref{estimate-of-product} to $\left \|\prod_{j=1}^N I_{m_j}(g^{(j)}_n)_T-\prod_{j=1}^N I_{m_j}(g^{(j)}_m)_T \right  \|_{L^2(\Pbb)}$ we have again
convergence to zero by  dominated convergence which holds thanks to \eqref{L2-condition}. Repeating the arguments of  Step 1 for the present situation it follows that 
$\prod_{j=1}^N I_{m_j}(f^{(j)})_T  \in L^2(\Pbb)$  and that \eqref{product-formula} holds.  \smallskip \\
The relation \ref{expectation-product-formula} follows immediately from \ref{product-formula} by taking $k=0.$
\end{proof}
Similar to \cite[Theorem 2.2.] {DoPecc18}  which concerns the product of $2$ multiple integrals, we expect that there exists  an if and only if  relation between  $\prod_{j=1}^N I_{m_j}(f^{(j)}  ) \in L^2(\Pbb)$ and the existence of the product formula. This was already addressed by Surgailis in  \cite{S84}.

\begin{conjecture} \label{conjecture}
Let $f^{(j)} \in L^2_{m_j}$ for $j=1,...,N.$
$$ \prod_{j=1}^N I_{m_j}(f^{(j)})_T \quad \text{  is square integrable}$$
$$ \iff$$
\equa
 \sum_{\begin{subarray}{c}|l|=k,\\ (l,l^o) \in D_N \end{subarray}}   \frac{1}{l!\, l^o !}  \,\,
  \widehat \bigstar^l_{l^o}  (f^{(1)},...,f^{(N)} )  \in L^2_k,  \quad \text{for each} \,\,\ k = 1,... ,m_1+...+ m_N,
  \tion
and the expression is finite for $k=0$. Moreover, the product formula  \eqref{product-formula} holds.
\end{conjecture}

\section{Applications}\label{sec:appl}

\subsection{The expectation of a power of a stochastic integral}\label{subsec:cum}

We apply now \eqref{product-formula}
to obtain explicit formulas for  moments and  cumulants of $I_1(f)_T$. The approach in \cite[Chapter 7]{PT11}, considers the Brownian and Poisson setting separately and uses diagram formulae to   treat all kinds of power moments and products. This program can of course be conducted here as well -- however, for this application, we follow a more elementary way.

To point  out the connection between the sets $ {\bf e}_n$ and  the numbers  $l_n$ and $l^o_n$ more clearly, we will  use  the notation $l_{\bf s}$  or $l_{ {\bf e}_n}$  instead of  $l_n$  if ${\bf s}= {\bf e}_n$,
and similarly for $l_n^o.$ 

First we observe that to compute $\mathbb{E}I_1(f)^N_T$, for some  $N\geq2$  (which means we have  $m_1=\dotsb=m_N=1$), we need to extract the term for $k=0$ in \eqref{product-formula}.  This has the consequence  that: 
\begin{itemize}
\item  The identity $\widehat{\bigstar}_{l^o}^0(f,\dotsc,f)=\bigstar_{l^o}^0(f,\dotsc,f)$ holds, since no $z$-variables are involved in these expressions.
\item It is only necessary to consider tuples $(0,\dotsc,0,l^o_1,\dotsc,l^o_{s_N})$ in $D_N$  so that for each $j\in \{1,\dotsc,N\}$, we need $\sum_{{\bf s}\ni j}l^o_{\bf s}=m_j=1$.
\end{itemize}
The latter point implies that one can establish a bijection between the tuples in $D_N^*$, i.e.\ those tuples from $D_N$ such that $l_1=\dotsb=l_{s_N}=l^o_{\{1\}}=\dotsb=l^o_{\{N\}}=0$, and the partitions of $\{1,\dotsc,N\}$ of block size $\geq2$. (We call the elements of the partition blocks and their cardinality block size.) We consider this set of partitions, since by \eqref{cen}, blocks of size $1$ do not yield a term in the product formula for $k=0$. Note that this bijection is one particular case of those that are analyzed and numbered by diagrams and multigraphs in \cite{PT11}.

 We thus get
\begin{align*} 
\{\text{Partitions of block size }\geq 2\}\to D_N^*,\quad P\mapsto \one_P=
\sum_{{\bf s}\in P}\one_{\{{\bf s}\}}.
\end{align*}

For one such partition consisting of $q$ sets excluding singletons,  $\{{\bf s}_1,\dotsc,{\bf s}_q\}$, we denote the block sizes by $p_i:=|{\bf s}_i|$. To compute $\bigstar_{l^o}^0(f,\dotsc,f)$, we have to investigate the variables $(z_{1:l_n^o}^{{\bf e}_n,0})$ that arise in the product $\prod_{j=1}^Nf\left((z_{1:l_n^o}^{{\bf e}_n,0})_{1\leq n\leq s_N, j\in {\bf e}_n}\right)$. 
 For a given partition $\{{\bf s}_1,\dotsc,{\bf s}_q\}$,
 if $j$ is contained in ${\bf s}_i$, then the same variable $z_j^{{\bf s}_i,0}$ appears in $p_i$ of the product's factors, and altogether there can be only $q$ different variables. Hence we obtain

\begin{align*}
\bigstar_{l^o}^0(f,\dotsc,f)=\prod_{i=1}^q \int_{(0,T]\times\mathbb{R}}f(z)^{p_i} c_{{\bf s}_i}(z)\m(dz),  
\end{align*}
and, using \eqref{cen} again for the sets of block size $2$,

\begin{align}\label{bigstarsamef}
\bigstar_{l^o}^0(f,\dotsc,f)=\bigg(\int_{(0,T]\times\mathbb{R}}f(z)^2\m (dz)\bigg)^{|\{i:p_i=2\}|}\cdot \prod_{p_i>2} \int_{(0,T]\times{\mathbb{R}_0}}f(z)^{p_i} \m(dz).
\end{align}
To find how many partitions deliver such a term, we   first consider  the number of all  partitions   of $\{1,\dotsc,N\}$ into $q$ blocks of size $p_1,\dotsc,p_q$  (that is, without the restriction $p_i\geq2$).

If we assume that there are  $j_a$ blocks of size $a$ in the partition, we have the relationships $j_1+j_2+\dotsb+j_{N-q+1}=q$ (and there cannot be more than $N-q+1$), $j_a=|\{i:p_i=a\}|$ and $p_1+\dotsb+p_q=j_1+2j_2+\dotsb+(N-q+1)j_{N-q+1}=N$.  Then the number of such partitions is given by the coefficient $b_{N,(j_1,\dotsc,j_{N-q+1})}$ of the monomial $X_1^{j_1}\cdots X_{N-q+1}^{j_{N-q+1}}$ in the $N$-th partial exponential Bell polynomial $B_{N,q}(X_1,\dotsc,X_{N-q+1})$  (see, for example, \cite[Definition 11.2]{Chara} or \cite[Definition 2.2.1]{PT11}). In particular, the coefficient is
$$b_{N,(j_1,\dotsc,j_{N-q+1})}=\frac{N!}{j_1!(1!)^{j_1}j_2!(2!)^{j_2}\dotsc j_{N-q+q}!((N-q+1)!)^{N-q+1}}.$$
The polynomial containing the information  about  all partitions i.e.with arbitrarily many blocks is the (exponential) Bell polynomial $B_N(X_1,\dotsc,X_N):=\sum_{q=1}^NB_{N,q}(X_1,\dotsc,X_{N-q+1})$. \\
Since we exclude partitions with block size $1$, we only have to consider the polynomial $B_{N}(0,X_2,\dotsc,X_N)$. In \eqref{bigstarsamef}, each  $p_i$-integral factor appears $j_{p_i}$ times.  Therefore, we get the following relations:

\begin{proposition}  \label{Bell-prop} {\color{white} .}
\begin{enumerate}[(i)]
\item \label{our-moments} Let $f \in  L^2_1\cap L^N_1.$  Then
\begin{align*}
\hspace*{-2em} \mathbb{E}I_1(f)_T^N = B_N\left(0,\int_{(0,T]\times\mathbb{R}}f(z)^2\m(dz),\int_{(0,T]\times{\mathbb{R}_0}}f(z)^{3} \m(dz),\dotsc,\int_{(0,T]\times{\mathbb{R}_0}}f(z)^{N} \m(dz)\right),
\end{align*}
or more in detail, 
\begin{align*}
&\hspace*{-2em} \mathbb{E}I_1(f)_T^N=\\
&\hspace*{-2em} \sum_{q=1}^N\sum_{\substack{\sum_{i=2}^{N-q+1}j_i=q\\ \sum_{i=2}^{N-q+1} ij_i=N}}  \!\!\!\!b_{N,(j_1,\dotsc,j_{N-q+1})}\left (\int_{(0,T]\times{\mathbb{R}}}\!\!\!\!f(z)^{2} \m(dz)\right )^{j_2}\!\!\cdots \left (\int_{(0,T]\times{\mathbb{R}_0}} \!\!\!\!f(z)^{N-q+1} \m(dz) \right )^{j_{N-q+1}}\hspace*{-0.5em}.
\end{align*}

\item \label{Brown-moments} For the Brownian case we recover the known relation
\begin{align*}
\mathbb{E}I_1(f)_T^N=B_N\left(0,\int_{(0,T]}f(s)^2\sigma^2ds,0,\dotsc,0\right)= (N-1)!!\Big(\int_{(0,T]}f(s)^2\sigma^2ds\Big)^{\frac{N}{2}}\one_{\{N\in 2\mathbb{N}\}}.
\end{align*}
\item \label{cumulants-relation} If $f \in  L^2_1\cap L^N_1$ for all $N \in \N$  then the  cumulants $\kappa_N$ of $I_1(f)_T$ are given by 
$$\kappa_1=0,   \quad \kappa_2=\int_{(0,T]\times{\mathbb{R}}}f(z)^{2}\m(dz), \,\, \kappa_N=\int_{(0,T]\times{\mathbb{R}_0}}f(z)^{N}\m(dz),  \,\, N\ge 3.$$

\end{enumerate}
\end{proposition}

\begin{proof}
\eqref{our-moments} is clear from the above considerations, \eqref{Brown-moments}  is obvious.  \eqref{cumulants-relation} one gets from \eqref{our-moments} by the formula relating cumulants and  moments (see  \cite[Corollary 3.2.2]{PT11}  or \cite{P16}). \end{proof}

\subsection{Expectations of products of stochastic integrals}
In the same way as before, we may consider  the product of integrals of different functions, $\mathbb{E}I_1(f^{(1)})_T\cdots I_1(f^{(N)})_T$. Using partitions as above, we can compute   $\bigstar_{l^o}^0(f^{(1)},\dotsc ,f^{(N)})$ by relating $l^o$ with the according partition $\{{\bf s}_1,\dotsc,{\bf s}_q\}$ without singletons and get
\begin{align*}
 \bigstar_{l^o}^0(f^{(1)},\dotsc, f^{(N)})=\prod_{i=1}^q\int_{(0,T]\times\R}\prod_{j\in {\bf s}_i}f^{(j)}(z)\m(dz).
\end{align*}
Taking all partitions into account, we obtain
\begin{align*}
&\mathbb{E}I_1(f^{(1)})_T\dotsc I_1(f^{(N)})_T \\
&=\sum_{\substack{P\in \mathcal{P}^*(N)\\P=\{{\bf s}_q,\dotsc,{\bf s}_q\}}}
                        \left (  \prod_{\substack{i=1 \\  |{\bf s}_i|=2}}^q\int_{(0,T]\times\R}\prod_{j\in {\bf s}_i}f^{(j)}(z)\m(dz)\right) 
                         \left (  \prod_{\substack{i=1 \\  |{\bf s}_i|>2}}^q\int_{(0,T]\times\R_0}\prod_{j\in {\bf s}_i}f^{(j)}(z)\m(dz)\right).
\end{align*}


\subsection{Long time behaviour and limit theorems}

In this section, which is inspired by \cite[\S4]{LPST14}, we are going to address the long time behaviour of integrals of the form $I_1(f)_T$ and, combining \eqref{product-formula} with the method of moments and cumulants (see \cite[\S A.3]{PN12}), we deduce a central limit theorem (from now on CLT) for $T\rightarrow+\infty$ in some special cases. 
 
Using  the properties of the cumulants and setting $\widetilde I_1(f)_T:= \frac{I_1(f)_T}{\Ebb[I_1\p2(f)_T]\p{1/2}}$ we get  from Proposition \ref{Bell-prop}  \eqref{cumulants-relation} that
\[
\kappa_N\big(\widetilde I_1(f)_T\big)=\begin{cases}\displaystyle1,&\quad N=2,\\
\frac{\int_{(0,T]\times\Rbb_0}f\p N(z)\m(\rmd z)}{\left(\int_{(0,T]\times\Rbb}f\p 2(z)\m(\rmd z)\right)\p{N/2}},&\quad N\geq3.
\end{cases}
\] 
{ So, if $f$ is measurable on $[0,\infty)\times\R$ and $f\in L^N([0,T]\times\R,\m)$ for all $N\geq 1$ and for every $T>0$
satisfies also 
\begin{equation}\label{eq:clt}
\lim_{T\rightarrow+\infty} \frac{\int_{(0,T]\times\Rbb_0}f\p N(z)\m(\rmd z)}{\left(\int_{(0,T]\times\Rbb}f\p 2(z)\m(\rmd z)\right)\p{N/2}}=0,\quad N\geq3,
\end{equation} 
then, by \cite[Theorem 1]{Ja88} we get that $\widetilde I_1(f)_T$ converges in distribution to $X\sim\Nscr(0,1)$ as $T\rightarrow+\infty$. Note that for $f\geq0$ condition \eqref{eq:clt} can be reformulated in terms of norms.}

We now consider functions $f$ of the form $f(z)=f(t,x)=g(t)h(x)$, where  $g: [0,\infty) \to \R$ satisfies $\int_0^T |g(t)|^N \rmd t < \infty$ for all $T>0$ and $h\in L\p N(\rho)$ for every $N\geq1.$ We set  $\rho(\rmd x)=\sig\p2\delta_0(\rmd x)+\nu(\rmd x)$. We assume $\|h\|_{L\p2(\rho)}=1$ and denote $\nu(h\p N):=\int_{\Rbb_0}h\p N(x)\nu(\rmd x)$. In this special case, we have
\begin{equation*} 
\kappa_N\big(\widetilde I_1(f)_T\big)=\nu(h\p N)\frac{\int_0\p T g\p N(t)\rmd t}{\left(\int_0\p Tg\p 2(t)\rmd t\right)\p{N/2}},\quad N\geq3.
\end{equation*}
Therefore, a sufficient condition for the CLT for $\widetilde I_1(f)_T$, is
\begin{equation}\label{eq:con.clt.prod}
\lim_{T\rightarrow+\infty}\frac{\int_0\p T g\p N(t)\rmd t}{\left(\int_0\p Tg\p 2(t)\rmd t\right)\p{N/2}}=0,\quad N\geq3.
\end{equation}
Condition \eqref{eq:con.clt.prod} is trivially satisfied if $g$ is a polynomial function or if it is the product of a polynomial function with a rapidly decaying function as, for example, if $g(t)=t\p m\rme\p{-t\p2/2}$.
Moreover, if $g$ is constant, which means that $(\widetilde I_1(f)_T)_{T\ge 0} $ is a L\'evy process,  \eqref{eq:con.clt.prod} is always satisfied.

More interesting is the case $g(t)=\rme\p{ t\p\al}$, with $\al>0$, that we illustrate in the following lemma.
\begin{lemma}\label{lem:exp.al}
Let $g(t)=\rme\p{t\p\al }$, $\al>0$. We then have
\[
\lim_{T\rightarrow+\infty} \frac{\int_0\p T\rme\p{Nt\p\al}\rmd t}{\left(\int_0\p T\rme\p{2t\p\al}\rmd t\right)\p{N/2}} =\begin{cases}
0,&\quad\al<1,\\
\frac{2\p{N/2}}{N},&\quad\al=1,\\
+\infty,&\quad{\al>1},
\end{cases}\quad N\geq3.
\]
\end{lemma}
\begin{proof}
We use  the  variable transform $u = t^\al$  to rewrite the expression 
$$\frac{\int_0\p T\rme\p{Nt\p\al}\rmd t}{\left(\int_0\p T\rme\p{2t\p\al}\rmd t\right)\p{N/2}}  =  \frac{ \frac{1}{\al} \int_0^{T^\al}\rme\p{Nu}  u^{\frac{1}{\al}-1}   \rmd u}
{\left( \frac{1}{\al}     \int_0^{T^\al}\rme\p{2u}   u^{\frac{1}{\al}-1}    \rmd u\right)\p{N/2}}.
 $$
Since for all $k\in \N$ it holds  $ \lim_{x \to +\infty} \frac{k \int_0^xe^{ku}  u^{\frac{1}{\al}-1}   \rmd u}{e^{kx}  x^{\frac{1}{\al}-1}  }=1,$ we get 
\equa
 \lim_{T \to +\infty}  \frac{   \frac{1}{\al}   \int_0\p {T\p\al}\rme\p{Nu}  u^{\frac{1}{\al}-1}   \rmd u}
{\left( \frac{1}{\al}     \int_0\p{T\p\al}\rme^{2u}   u^{\frac{1}{\al}-1}    \rmd t\right)\p{N/2}}
&=&   \lim_{T \to +\infty}   \frac{  \frac{1}{N\al }   \rme^{NT^\al}  T^{1-\al} }{   \left (   \frac{1}{2 \al}    \rme^{2T^\al }  T^{1-\alpha} \right)^{N/2} } 
=
\begin{cases}
0,&\quad\al<1,\\
\frac{2\p{N/2}}{N},&\quad\al=1,\\
+\infty,&\quad{\al>1},
\end{cases}
 \tion
 
\end{proof}
The following theorem is an immediate consequence of Lemma \ref{lem:exp.al}.
\begin{theorem}\label{thm:clt}
Let $f(t,x)=h(x)\rme\p{t\p\al}$, where $h\in L\p N(\rho)$ for every $N\geq1$ and $\|h\|_{L\p2(\rho)}=1$. We have:

\textnormal{(i)} If $\al<1$, $\widetilde I_1(f)_T$ converges in law as $T\rightarrow+\infty$ to a standard normal distributed random variable.

\textnormal{(ii)} If $\al>1$, $\widetilde I_1(f)_T$ cannot converge in law as $T\rightarrow+\infty$ to a random variable whose distribution is determined by its moments.
\end{theorem}
In the next proposition  we discuss the asymptotic behaviour of the stochastic integral in Theorem \ref{thm:clt} for $\al=1$.  We denote the Lebesgue measure on $\Rbb$  and on $\Rbb_+$ by $\lm$  and $\lm_+$, respectively.

\begin{proposition}\label{prop:clt.al1} 
Let  $\Pi$ be a Poisson random measure on $\Rbb_+\times\Rbb\times\Rbb_0$ with the compensator $\lm_+\otimes\lm\otimes\nu$ and let the function $h\geq0$ be such that $\nu(h\p2)=1$ and $\nu(\rme\p{u h})<+\infty$ for every $u\leq\varepsilon$, for a fixed $\varepsilon>0$.  We set ${a(x,y)}:=h(x)2\p{y/2}$, $E:=[0,\frac12\ln2]\times(-\infty,1]\times\Rbb_0$ and define 
$$Y:=\int_E{a(x,y)}(\Pi-\lm_+\otimes\lm \otimes\nu){ (\rmd t,\rmd y,\rmd x)}.$$

 We then have:

\textnormal{(i)} $\Ebb[\rme\p {u Y}]<+\infty$, for every $u\leq\frac{\varepsilon}{\sqrt2}$. Hence the distribution $\Pbb_Y$ of $Y$ is determined by its moments.

\textnormal{(ii)} The sequence of the cumulants of $Y$ is given by $k\p Y_1=0$ and $k_N\p Y=\nu(h\p N)\frac{2\p{N/2}}N$, $N\geq2$.

\textnormal {(iii)} Let $f(t,x)=h(x)\rme\p{t}$. Then $\widetilde I_1(f)_T$ converges in distribution to $Y$ as $T\rightarrow+\infty$.
\end{proposition} 
\begin{proof}
Since the condition on the convergence of the cumulants is equivalent to the condition on the convergence of the moments, (iii) follows from (i), (ii) and Lemma \ref{lem:exp.al} by \cite[Theorem 30.2]{B08}. To see (i) and (ii), we introduce the nonnegative random variable $ Z:=\int_E{ a(x,y)} \Pi{(\rmd t,\rmd y,\rmd x)}$
and notice that, for every $0<u\leq\frac{\ep}{\sqrt2}$,  from  \cite[Theorem 3.9 and Exercise 3.4]{LP17}, we get 
\[
\begin{split}
\Ebb\big[\rme\p{u Z}\big]&=\exp\bigg(\int_0\p{\frac12\ln 2}\int_{-\infty}\p1\int_{\Rbb_0}(\rme\p{u a(x,y)}-1)\nu(\rmd x)\rmd y\rmd t\bigg)
\leq\exp\big(\nu(\rme\p{\sqrt 2uh}-1)\big)<+\infty,
\end{split}
\]
where, to get the first estimate, we expand the integrand as a power series.
This together with the definition of $Y$ yields
\[
+\infty>\Ebb\big[\rme\p{uY}\big]=
\exp\bigg(\sum_{N=2}\p{+\infty}\frac{u\p N}{N!}\nu(h\p N)\frac{2\p{N/2}}{N}\bigg),\quad 0<u\leq\frac{\ep}{\sqrt2}.
\]
Thus, (i) and (ii) are shown.
 \end{proof}

\section{The number of elements  in  \texorpdfstring{$D_N$}{Dssn}  if \texorpdfstring{$|l|=k$}{l=k}}\label{sec:DSSn}

Without handling the enumeration of appearing terms, it is not clear how to perform the calculations on a computer. The formulas below allow for recursive as well as iterative implementations.
Here we use again the notation  $l^o_{\{j\}}$    or $l^o_{ {\bf e}_n}$ for   $l^o_n$   if $ {\bf e}_n ={\{j\}}$.
Counting the number of contributing terms in \eqref{product-formula} for a given $|l|=k\leq m_1+\dotsb+m_N$, note that for any $ j\in \{1,\dotsc,N\}$ it holds $l^o_{\{j\}}=0$. This is  because of  definition \eqref{cen} (stating that the $c_{\{j\}}$ are zero). Thus, we will omit those components. \smallskip

We will tackle the problem in the subsections below in 3 ways, allowing a variety of approaches to calculate the number on a computer: by deriving a recursion formula, by using weak compositions, and by applying generating functions. Of course, it is also possible to count and study the set $D_N$ through other combinatorial approaches like numbering diagrams and multigraphs. This has been done in \cite{PT11} for the Brownian and pure-jump setting separately. The same program as there can be conducted for our case where the Brownian part and the jump part are  combined and is subject to  ongoing research. Studies in combinatorical details of this very direction ought to deliver further insights into the structure of integrands appearing in the  product formula.

\subsection{... by a recursion formula} 

\begin{proposition} 
We have the following recursion formula for 
 $$|\{ (l,l^o) \in   D_N:  |l|=k \text { and }  l^o_{\{1\}}=...=l^o_{\{N\}}= 0\}|=:  {\tt C}(k,N,(m_1,\dotsc,m_N,0,\dotsc)),$$
where 
\begin{align*}
 {\tt C}(k,\tilde{N},(m_1, \dotsc, m_N,0,...)):=0\quad\text{if }\tilde{N}\neq N,
\end{align*}
\equal \label{not-counted}  {\tt C}\big(k,N,(m_1,\dotsc,m_N,0,\dotsc)\big):=0 \quad  \text{ if } \exists k \in \{1,...,N\} \text{ with}  \quad m_k<0.
 \tionl
$$
 {\tt C}\big(0,0,(0,0,\dotsc)\big):=1,
$$ 

and
\begin{align*}
& {\tt C}\big(k,N,(m_1,\dotsc,m_N,0,\dotsc)\big) \\&=\!\!\!\sum_{\kappa=0}^{\min\{k,m_N\}}\!\!\!\sum_{\substack{l_{\{N\}}+\dotsb+l_{\{1,\dotsc,N\}}\\=\kappa}}\ \ 
\sum_{\substack{l^o_{\{1,N\}}+\dotsb+l^o_{\{1,\dotsc,N\}}\\=m_N-\kappa}} \!\!\! {\tt C}\biggl(k-\kappa,N-1,(\hat m_1(l,l^o),...,\hat m_{N-1}(l,l^o),0, \dotsc)\biggr) \\\
\end{align*}
where 
\begin{align*}
&\hat m_r(l,l^o)
:= m_r-\sum_{\substack{{\bf s}\in \mathcal{P}(\{1,\dotsc,N-1\})\\ r\in \{N\}\cup {\bf s}}}(l_{\{N\}\cup {\bf s}}+l^o_{\{N\}\cup {\bf s}}),
\end{align*}
and   $\mathcal{P}(\{1,\dotsc,N-1\})$ stands for the power set of $\{1,\dotsc,N-1\}.$
\end{proposition}

\begin{proof}
The definition of  $D_N$ (see \eqref{Dsn}) requires (we use here $l_{{\bf e}_n}:=l_n$)
$$\sum_{n=1}^{\ssn} (l_{{\bf e}_n}+ l_{{\bf e}_n}^o) \one_{{\bf e}_n}(j) =m_j \quad  \text{for all} \quad j=1,...,N.$$
Ordering  the index set such that all ${\bf e}_n$ containing $N$ are behind the other sets that do not contain $N$, we  observe the tuples
$$(l_{\{1\}},\dotsc,l_{\{1,\dotsc,N-1\}},l_{\{N\}},\dotsc,l_{\{1,\dotsc,N\}}, l^o_{{\bf e}_{N+1}},\dotsc,l^o_{{\bf e}_{\ssn}}).$$
We put $\kappa:=|(l_{\{N\}},\dotsc,l_{\{1,\dotsc,N\}})|$ and know that 
$$\kappa+|(l^o_{\{1,N\}},\dotsc,l^o_{\{N-1,N\}}\dotsc,l^o_{\{1,\dotsc,N\}})|=m_N \quad \text{ and} \quad \kappa\leq\min\{k,m_N\}. $$
For every  choice  of $(l_{\{N\}},\dotsc,l_{\{1,\dotsc,N\}})$ and $(l^o_{\{1,N\}},\dotsc,l^o_{\{N-1,N\}},\dotsc,l^o_{\{1,\dotsc,N\}})$
the number of possible choices for the remaining  $(l_{\{1\}},\dotsc,l_{\{1,\dotsc,N-1\}},l^o_{\{1,2\}},\dotsc,l^o_{\{N-2,N-1\}},\dotsc,l^o_{\{1,\dotsc,N-1\}} )$  is determined by 
 \equa
 && \biggl(m_1-\!\!\sum_{\substack{{\bf s}\in \mathcal{P}(\{1,\dotsc,N-1\})\\ 1\in \{N\}\cup {\bf s}}}(l_{\{N\}\cup {\bf s}}+l^o_{\{N\}\cup {\bf s}}),\ \dotsc\ ,\ m_{N-1}-\sum_{\substack{{\bf s}\in \mathcal{P}(\{1,\dotsc,N-1\})\\ N-1\in \{N\}\cup {\bf s}}}(l_{\{N\}\cup {\bf s}}+l^o_{\{N\}\cup {\bf s}})\biggr) \\
 &=& (\hat m_1(l,l^o),..., \hat m_{N-1}(l,l^o))
 \tion 
 with $k$ replaced by  $k-\kappa$, which implies the claim.  Clearly, only those $(l,l^o)$ can be chosen which
 satisfy $\hat m_r(l,l^o) \ge 0$ for all $r=1,..., N-1$ because $\hat m_r(l,l^o)$ is the number of 'unused' variables of the kernel in the $r$-th factor of the product. Therefore we use \eqref{not-counted}.
\end{proof}

\bigskip

Concerning our special cases considered in Remark \ref{en} the following holds
\begin{rem} \label{numbers} 
\begin{enumerate}[(a)] 
\item \label{jump-number}  For the jump case $M(\rmd t, \rmd x) =  \tilde N(\rmd t,\rmd x)$  
the number of appearing terms is $ {\tt C}(k,N,(m_1,\ldots,m_N,0,...))$.
\item\label{cbrown-number} For the Brownian motion case  $M(\rmd t, \rmd x) = \sigma dW_t \delta_0(\rmd x)$ 
 the number of appearing possible nonzero terms is significantly less and can be -- in the same way as before -- recursively enumerated by the formula
\begin{align*}
& {\tt C}_W\big(k,N,(m_1,\dotsc,m_N,0,\dotsc)\big)=\\
&\sum_{\kappa=0}^{\min\{k,m_N\}}\!\!\!\!\! \sum_{\substack{l^o_{\{1,N\}}+\dotsb+l^o_{\{N-1,N\}}=m_N-\kappa\\ l^o_{\{i,j\}}\geq 0,\ i\neq j,}}\!\!\!\!\!\!\!\!\!\!\!\!\!\!\!\!  {\tt C}_W\big(k-\kappa,N-1,( m_1-l^o_{\{1,N\}},\ \dotsc\ ,\ m_{N-1}-l^o_{\{N-1,N\}},0,...)\big),
\end{align*}
setting $ {\tt C}_W$ to zero or one for all other cases, as before.
\end{enumerate}
\end{rem}

\bigskip 
\subsection{... by counting weak compositions}

To count those tuples without using a recursion, the following formula for tuples $(l_u)_{u \in U}$ where $U$ is  a finite set  and  $l_u \in \{0,1,\dotsc\}$, involving weak compositions is a help. A weak composition of $r \in \N \cup \{0\}$  into
$n \in \N$ parts  is any representation $r = a_1+...+a_n$ with $a_1,...,a_n \in  \N \cup \{0\}.$  If we  denote by
$\begin{pmatrix}\!\!\begin{pmatrix}r\\ n\end{pmatrix}\!\!\end{pmatrix}$ the number of weak compositions of $r$ into $n$ parts, it holds that $\begin{pmatrix}\!\!\begin{pmatrix}r\\ n\end{pmatrix}\!\!\end{pmatrix}=\dbinom{r+n-1}{n-1}$.
For example, the weak compositions of $r=5$ into  $n=2$ parts are given by
$$0+5,\,\,\, 5+0, \,\,\, 1+4, \,\,\,4+1, \,\,\, 2+3, \,\,\,3+2. $$
\newpage
\begin{lemma}\label{lem:tupleenum}
Let $A_1,\dotsc,A_N$ be subsets of a finite set $U$ such that $\bigcup_{i=1}^N A_i=U$, $A_j\cap\bigcap_{\substack{i=1\\ i\neq j}}^NA_i^c\neq \emptyset$, for all $j=1,\dotsc,N$, and let numbers  $m_1,\dotsc,m_N \in \N$ be given.
For ${\bf t}=(t_1,...,t_N) \in \{0,1\}^N$ we set
\equa A_{(t_1,...,t_N)} &:=&\left (\bigcap_{i=1, t_i=1}^N    A_i \right )  \cap  \left (\bigcap_{j=1, t_j=0}^N    A_j^c \right )\\
 \min_{\bf t} \, m &:=& \min\{ m_i: t_i =1, i=1,...,N\}.
 \tion

Let $\mathfrak{n}\colon\{0,1\}^N\setminus\{(0,\dotsc,0)\}\leftrightarrow\{1,\dotsc,2^N\}$ be a bijection such that $$\mathfrak{n}\big((1,\dotsc,1)\big)=1, \,\,\, \mathfrak{n}\big((0,\dotsc,0,1)\big)=2^N$$ and $\mathfrak{n({\bf t})<\mathfrak{n}({\bf s})}$ whenever ${\bf s}>{\bf t}$ (where by ${\bf s}>{\bf t}$ we mean that   $s_i \ge t_i$ for all $i=1,...,N,$ and there exists at least one $i_0$ such that 
 $s_{i_0}=1$ and  $t_{i_0} =0$).

For shortness, with  a slight abuse of notation we identify $q_{\mathfrak{n}({\bf t})}$ and $q_{\bf t}$ (so e.g. $q_{(1,\dotsc,1)}=q_1$), and we abbreviate,
\begin{align*}
\min_{\bf t} \, m - q_{>} := \min\bigg\{ m_i-\sum_{\substack{\iota=1\\\mathfrak{n}^{-1}(\iota)(i)=1}}^{\mathfrak{n}({\bf t})-1}q_{\iota}: \, t_i =1, i=1,...,N\bigg\}.
\end{align*}
Then the cardinality of
\begin{align*}
\biggl\{(l_{u})_{u\in U}: \sum_{a\in A_i}l_ a=m_i, \, i=1,\dotsc,N\biggr\}
\end{align*}
is given by 
 \begin{align*}
&\sum_{q_1=0}^{\min_{\mathfrak{n}^{-1}(1)} m -q_>} \begin{pmatrix}\!\!\begin{pmatrix}q_1\\ |A_{\mathfrak{n}^{-1}(1)}|\end{pmatrix}\!\!\end{pmatrix}\cdots\sum_{q_j=0}^{\min_{\mathfrak{n}^{-1}(j)} m -q_>} \begin{pmatrix}\!\!\begin{pmatrix}q_j\\ |A_{\mathfrak{n}^{-1}(j)}|\end{pmatrix}\!\!\end{pmatrix}\cdots\\
&\qquad\cdots\sum_{q_{2^N-N}=0}^{\min_{\mathfrak{n}^{-1}(2^N-N)} m -q_>} \begin{pmatrix}\!\!\begin{pmatrix}q_{2^N-N}\\ |A_{\mathfrak{n}^{-1}(2^N-N)}|\end{pmatrix}\!\!\end{pmatrix}\cdot\prod_{|{\bf t}|=1}\begin{pmatrix}\!\!\begin{pmatrix}\min_{\bf t} m-q_>\\ |A_{\bf t}|\end{pmatrix}\!\!\end{pmatrix},
\end{align*}
and whenever a set $A_{\bf s}$ is empty, the according summation does not appear in the above formula and $q_{\bf s}$ is then set to zero.
\end{lemma}
\begin{proof}
The formula can be seen by partitioning $U$ into all possible pieces emerging from intersecting the sets $A_1,\dotsc,A_N$.  Provided that $ \bigcap_{i=1}^N A_i\neq \emptyset$, one summand contained in each of the sums  $\sum_{a\in A_i}l_a$  is 
$\sum_{a\in \bigcap_{i=1}^N A_i}l_a= \sum_{a\in  A_{(1,...,1)}}l_a     = q_{(1,...,1)}=q_1$ (an intersection with all sets involved).
The possibilities for $q_1$ are $0,...,\min_{(1,...,1)} m$ (which equals $\min_{(1,...,1)} m -q_>$), and the number of tuples $(l_a)_{a\in A_{(1,...,1)}}$ that sum up to $q_1$ are given by
$\begin{pmatrix}\!\!\begin{pmatrix}q_1\\ |A_{(1,...,1)}|\end{pmatrix}\!\!\end{pmatrix}=\begin{pmatrix}\!\!\begin{pmatrix}q_1\\ |A_{\mathfrak{n}^{-1}(1)}|\end{pmatrix}\!\!\end{pmatrix}.$  
For the second sum,  as we already decided for a part of that sum to be $q_1$, for $q_2=q_{(1,\dotsc,1,0,1,\dotsc,1)}$  there are possibilities from $0$ to $\min_{(1,\dotsc,1,0,1,\dotsc,1)} m -q_1$, which equals $\min_{(1,\dotsc,1,0,1,\dotsc,1)} m -q_>$ again. For $q_3$ and the according tuple $\mathfrak{n}^{-1}(3)$, we have the upper limit $\min\{m_i-q_1-q_2\cdot\one_{ \{\mathfrak{n}^{-1}(2)(i) =1\}} : \mathfrak{n}^{-1}(3)(i)=1,\,  i=1,\dotsc,N\}$,  which is $\min_{\mathfrak{n}^{-1}(3)} m -q_>$, and the number of tuples is given by $\begin{pmatrix}\!\!\begin{pmatrix}q_3\\ |A_{\mathfrak{n}^{-1}(3)}|\end{pmatrix}\!\!\end{pmatrix}$.  We proceed nesting the sums until we reach the number  $2^N-N$, which means that from thereon, the corresponding tuples $\mathfrak{n}^{-1}(j), \,  j>2^N-N,$ have only one nonzero element.

For those last $q_{2^N-N+l}, l=1,\dotsc,N$, and the according tuples, there are no further subsets of our partition left. Taking, without loss of generality, the tuple $(1,0,\dotsc,0)$, the remaining difference of the overall sum of the $q$'s to $m_1$, which is 
$$m_1-\sum_{\substack{\mathfrak{n}({\bf s})< \mathfrak{n}((1,0,...,0))\\{ {\bf s}_1=1}}}q_{\bf s}=\min_{(1,0,\dotsc,0)}m -q_>,$$ needs only to be split up between the components of $A_{(1,0,\dotsc,0)}$, so no further summation appears in this case.

Also, for all other sets of type $A_1^c\cap\dotsb\cap A_{j-1}^c\cap A_j\cap A_{j+1}^c\cap\dotsb A_N^c$, no summation appears since there is then only one possible choice for 
 $$q_{(0,\dotsc,0,\underbrace{1}_{j\text{-th component}},0,\dotsc,0)},$$ namely $m_j-\sum_{\substack{\mathfrak{n}({\bf s}) <\mathfrak{n}((0,\dotsc,0,1,0,\dotsc,0))\\ {\bf s}_j=1}} q_{\bf s}$.
\end{proof}

In our situation of multiple integrals, to use the above formula e.g. to count the admissible tuples $(l^o_{{\bf e}_{N+1}},\dotsc,l_{{\bf e}_{s_N}}^o)$, our notation corresponds to  setting in the above Lemma $U:=\{B\subseteq\{1,\dotsc,N\}: |B|>1\}$ and $A_i:=\{B\subseteq \{1,\dotsc,N\}: |B|>1, i\in B\}$ for $i=1,\dotsc,N$. Then, 
$\bigl|\bigcap_{i=1, t_i=1}^NA_i\cap \bigcap_{j=1, t_j=0}^NA_j^c\bigr|=1$ 
for all {\bf t}, because $\bigcap_{i=1, t_i=1}^NA_i\cap \bigcap_{j=1, t_j=0}^NA_j^c$ contains exactly one set, namely $\{i:t_i=1\},$ and the formula reduces to
\begin{align} \label{hat-M}
& \hat{\tt C}(m_1,\dotsc,m_N):=\Big |\Big\{(l^o_{{\bf e}_{N+1}},\dotsc,l^o_{{\bf e}_{s_N}}): \forall i\in \{1,\dotsc,N\}:\sum_{{\bf e}_j\ni i} l^o_{{\bf e}_j} =m_i\Big\} \Big |\\
&=\sum_{q_1=0}^{\min_{(1,...,1)} m}1\dotsm
\sum_{q_j=0}^{\min_{\mathfrak{n}^{-1}(j)} m -q_{_{>}}}1 \dotsm  \sum_{q_{2^N-N}=0}^{\min_{\mathfrak{n}^{-1}(2^N-N)} m -q_{_{>}}} 1. \notag
\end{align}

To count the admissible tuples $(l_{{\bf e}_1},\dotsc,l_{{\bf e}_{s_N}})$, set $A_i:=\{B\subseteq \{1,\dotsc,N\}:i\in B\}$ for $i=1,\dotsc,N$. We have the additional condition that $\sum_{i=1}^{s_N}l_{{\bf e}_i}=k$. Therefore, we have to slightly change the above formula, in the way, that for the last factor there is only one possibility if this condition is met, and none otherwise. Also, the minima in the upper sum limits must be bounded by terms involving $k$. Hence, we obtain

\begin{align} \label{hajek-M}
& \check{\tt C}(k; m_1,\dotsc,m_N):=\Big | \Big\{(l_{{\bf e}_{1}},\dotsc,l_{{\bf e}_{\sn}}): \sum_{i=1}^{s_N}l_{{\bf e}_i}=k \,\text{ and } \,\forall i\in \{1,\dotsc,N\}:\sum_{{\bf e}_j\ni i} l_{{\bf e}_j}=m_i\Big\} \Big |\\
&=\sum_{q_1=0}^{\min_{(1,...,1)} m \wedge k}1\cdot\sum_{q_2=0}^{(\min_{\mathfrak{n}^{-1}(2)} m -q_1)\wedge(k-q_1)} 1\cdot
\dotsm  
\sum_{q_j=0}^{(\min_{\mathfrak{n}^{-1}(j)} m -q_{_{>}})\wedge(k-\sum_{\iota=1}^{j-1}q_\iota)} 1 \dotsm   \notag\\
& \hspace{1em} \dotsm\sum_{q_{2^N-N}=0}^{(\min_{\mathfrak{n}^{-1}(2^N-N)} m -q_{_{>}})\wedge(k-\sum_{\iota=1}^{2^N-N-1}q_\iota)}\one_{\biggl\{m_1-q_{_{>}}=k-\sum_{\iota=1}^{2^N-(N-1)}q_\iota\biggr\}}\dotsm \one_{\biggl\{m_N-q_{_{>}}=k-\sum_{\iota=1}^{2^N}q_\iota\biggr\}}. \notag
\end{align}

To count then the elements in $D_N$ with $|l|=k$ we have the following corollary which follows directly from the definition  of $D_N$:

 \begin{corollary}[the general case]\label{cor:general-number-M}  Let $\hat{ {\tt C}}(m_1,\dotsc,m_N)$ be the number of tuples $(l^o_{{\bf e}_{N+1}},\dotsc,l^o_{{\bf e}_{s_N}})$ given by \eqref{hat-M}, and $\check{ {\tt C}}(k; m_1,\dotsc,m_N)$  the  number of tuples $(l_{{\bf e}_{1}},\dotsc,l_{{\bf e}_{s_N}})$ given by \eqref{hajek-M}, then,
 \begin{align*}
 & |\{ (l,l^o) \in   D_N:  |l|=k \text { and }  l^o_{\{1\}}= ...=  l^o_{\{N\}}=0\}|   \\
&= {\tt C}(k,N,(m_1,\dotsc,m_N,0,\dotsc))\\
&=\sum_{(\hat{m}_1,\dotsc,\hat{m}_N)+(\check{m}_1,\dotsc,\check{m}_N)=(m_1,\dotsc,m_N)}\!\!\!\!\!\!\!\!\!\!\!\!\!\!\!\!  \hat{\tt C}(\hat{m}_1,\dotsc,\hat{m}_N)   \check{\tt C}(k; \check{m}_1,\dotsc,\check{m}_N).
 \end{align*}

\end{corollary}

For the Brownian motion case  $M(\rmd t, \rmd x) = \sigma dW_t \delta_0(\rmd x)$  the formula simplifies:

\begin{corollary}
If we consider multiple integrals w.r.t.~the Brownian motion,  
\begin{align*}  &\hat{ {\tt C}}_W(m_1,\dotsc,m_N) \\ &:= \Big | \Big\{(l^o_{{\bf e}_{N+1}},\dotsc,l^o_{{\bf e}_{s_N}}): \forall i\in \{1,\dotsc,N\}:\sum_{{\bf e}_j\ni i} l^o_{{\bf e}_j} =m_i\Big\} \Big |\\
&=\sum_{q_{12}=0}^{m_1\wedge m_2}1\cdot\sum_{q_{13}=0}^{(m_1-q_{12})\wedge m_3}1
\cdots\sum_{q_{N-1\,N}=0}^{(m_{N-1}-q_{1\,N-1}-\dotsb-q_{N-2\,N-1})\wedge (m_N-q_{1\,N}-\dotsb -q_{N-2\,N})}1. \notag
\end{align*}
Further,
\begin{align*}
&\check{ {\tt C}}_W(k; m_1,\dotsc,m_N):=\Big |\Big\{(l_{{\bf e}_{1}},\dotsc,l_{{\bf e}_{N}}): \forall i\in \{1,\dotsc,N\}:\sum_{{\bf e}_j\ni i} l_{{\bf e}_j} =m_i, \sum_{j=1}^{s_N}l_{{\bf e}_j} =k\Big\} \Big |\\
&=\one_{\{m_1+\dotsb+m_N=k\}}.
\end{align*}
Then, following the above Corollary \ref{cor:general-number-M},
$$ {\tt C}_W(k,N,(m_1,\dotsc,m_N,0,\dotsc))=\Big | \{ (l,l^o) \in   D_N:  |l|=k \text { and } l^o_{\{j\}}= 0, j=1,...,N\} \Big |$$ can be calculated by
$$\sum_{\substack{(\hat{m}_1,\dotsc,\hat{m}_N)+(\check{m}_1,\dotsc,\check{m}_N)=(m_1,\dotsc,m_N)\\ \check{m}_1+\dotsb+\check{m}_N=k}}\!\!\!\!\!\!\!\!\!\!\!\!\!\!\!\!\hat{ {\tt C}}(\hat{m}_1,\dotsc,\hat{m}_N).$$
 \end{corollary}
\begin{proof}
For $\hat{ {\tt C}}_W$, the result follows by a similar use of Lemma \ref{lem:tupleenum} as above, setting
$A_i:=\{\{i,j\},j=1,\dotsc,N, j\neq i\}, i=1,\dotsc,N$.

Regarding the number $\check{ {\tt C}}_W(k; m_1,...,m_N)$, because of \eqref{e:cBrown}, only the first $N$ sets $l_{\{1\}},...,l_{\{N\}}$ are the support of such a contributing tuple $l$, and there is only one summand, $l_{\{j\}}$  to be summed up to $m_j$. Therefore, for given $(m_1,\dots,m_N)$, the only appearing tuple is $l=(m_1,\dotsc,m_N,0,\dotsc,0)$ and the resulting number is $$\check{ {\tt C}}_W(k; m_1,\dotsc,m_N)=\one_{\{m_1+\dotsb+m_N=k\}}.$$
\end{proof}

\subsection{... via generating functions}

Another way to count the elements of $D_N$ is by the help of generating functions (see \cite{Bona}, \cite{Wilf}).

Here we use  the  multiindex notation, in the sense that for a tuple $t\in \N^n$ and for variables $x=(x_1,\dotsc,x_n)$ we have $x^t=x_1^{t_1}\dotsm x_n^{t_n}$. For a set $S\subseteq\{1,\dotsc,N\}$, we set 

$$t_{S,i}:=\begin{cases}1,& i\in S,\\
0,& i\notin S\end{cases}$$
for $i=1,...,N.$
Hence we have $t_S \in \{0,1\}^N.$
As above we denote by $\hat{ {\tt C}}(m_1,\dotsc,m_N)$  the number of tuples $(l^o_{{\bf e}_{N+1}},\dotsc,l^o_{{\bf e}_{s_N}})$, and by $\check{ {\tt C}}(k,m_1,\dotsc,m_N)$ the number of tuples $(l_{{\bf e}_{1}},\dotsc,l_{{\bf e}_{s_N}})$  with $\sum_{i=1}^{s_N}l_{{\bf e}_i}=k$.
The generating function of $\hat{ {\tt C}}$ in the variables $x=(x_1,\dotsc,x_N)$ is given by the (formal) power series
\begin{align*}
g_{\hat{ {\tt C}}}(x)=\sum_{\hat{m}\in \N^N}\hat{ {\tt C}}(\hat{m})x^{\hat{m}}.
\end{align*}
We claim that
\begin{align*}
g_{\hat{ {\tt C}}}(x)=\sum_{\kappa=0}^\infty \,\,\,\,\sum_{l^o_{{\bf e}_{N+1}}+\dotsb+l^o_{{\bf e}_{s_N}}=\kappa} \,\,\,\left(x^{t_{{\bf e}_{N+1}}},\dotsc,x^{t_{{\bf e}_{s_N}}}\right)^{l^o},
\end{align*}
where $\left(x^{t_{{\bf e}_{N+1}}},\dotsc,x^{t_{{\bf e}_{s_N}}}\right)^{l^o}$ is the same as $\prod_{j=N+1}^{s_N} x^{l_{{\bf e}_j}\cdot\,t_{{\bf e}_j}}$.
 This equality can be seen by multiplying out the terms on the right hand side: The definition of the $t_S$ is constructed in the way such that 1 is added to the coefficient of $x^{\hat{m}}$, for each $l^o$ for which
\begin{align}\label{eq:sumcond}
 \sum_{\substack{j= N+1\\ i\in {\bf e}_j}}^{s_N}l^o_{{\bf e}_j}=\hat{m}_i,\quad i=1,\dotsc,N,
 \end{align}
 and the sum $\sum_{j= N+1}^{s_N}l^o_{{\bf e}_j}=\kappa$. Summing over all $\kappa$, we see that the last summation condition
 does not play a role and we are left with \eqref{eq:sumcond}, which is defining for $\hat{ {\tt C}}(\hat{m})$.

Note that the part
$\sum_{l^o_{{\bf e}_{N+1}}+\dotsb+l^o_{{\bf e}_{s_N}}=\kappa}\left(x^{t_{{\bf e}_{N+1}}},\dotsc,x^{t_{{\bf e}_{s_N}}}\right)^{l^o}$ can be obtained by first taking $\biggl(\sum_{j= N+1}^{s_N}Z_j\biggr)^{\kappa}$, for formal variables $Z=(Z_{N+1},\dotsc,Z_{s_N})$, then setting all (multinomial) coefficients that are not zero to 1 in the resulting polynomial. Finally, replace the $Z_j$ by the $x^{t_{{\bf e}_j}} = x_1^{t_{{\bf e}_j,1}}\cdots x_N^{t_{{\bf e}_j,N}}$.\smallskip

In the same way, the generating function of $\check{ {\tt C}}$ is

$$
g_{\check{ {\tt C}}}(y,x)=\sum_{k=0}^\infty\sum_{\check{m}\in \N^N}\check{ {\tt C}}(k,\check{m})y^k x^{\check{m}}=\sum_{{k}=0}^\infty y^k\sum_{l_{{\bf e}_1}+\dotsb+l_{{\bf e}_{s_N}}={k}}\left(x^{t_{{\bf e}_1}},\dotsc,x^{t_{{\bf e}_{s_N}}}\right)^{l},$$
where again $\sum_{l_{{\bf e}_1}+\dotsb+l_{{\bf e}_{s_N}}={k}}\left(x^{t_{{\bf e}_1}},\dotsc,x^{t_{{\bf e}_{s_N}}}\right)^{l}$ can be obtained by first taking $\biggl(\sum_{j=1}^{s_N}Z_j\biggr)^{k}$, then setting all coefficients that are not zero to 1 and, after that, inserting the $x^{t_{{\bf e}_j}}$ again.\smallskip

The generating function of ${\tt C}$, $g_{\tt C}(y,x)$ is then just $g_{\check{ {\tt C}}}(y,x)g_{\hat{ {\tt C}}}(x)$.


 \appendix \label{proofmain}
\section{ The  proof of Proposition \ref{alpha-products} \label{section-proof}}
 
We start with some preparations. We observe that in \eqref{random-deterministic}  the $ \al_{{\bf s}_n,i_n}(z_n)\cM^{ i_n}(dz_n) $  are   independent random  measures  if $i_n =1$ and
deterministic measures  if $i_n =0$, on
$((0,T]\times \R, \cB((0,T]\times \R)).$ 
We  define {\it mixed} multiple integrals to derive  a stochastic Fubini type relation for iterated  (mixed) integrals.

\begin{definition}[mixed multiple integrals] \label{mixed}
 Assume $f_1,...,f_k \in L^2_1.$  For $J \subseteq \{1,...,k\}$  and $n \in  \{1,...,k\}$ let
\equa
\mu_n(dz) := \left \{ \begin{array}{ll}  f^2_n(z) \m(dz) & \text{ if }  \, \,  n \in J   \\ \\
f_n (z) M(dz) & \text{ if } \,\ n \notin J.\\\\
        \end{array}      \right . 
\tion
Then we define  for any $g \in {\cal E}_k$ (see  \eqref{L0m}) the mixed multiple integral
$$I_{\mu_{1:k}}(g)_T :=  \sum_{j_1,...,j_k=1}^n  a_{j_1,...,j_k} \mu_1(A_{j_1})\cdots \mu_k(A_{j_k}). $$ 
\end{definition}

\begin{lemma} \label{mixed-properties} The mixed multiple integral $I_{\mu_{1:k}} $ from Definition \ref{mixed} has the following properties:
\begin{enumerate}[(i)]
\item \label{mixed-linear}    $I_{\mu_{1:k}}: {\cal E}_k \to L^2(\mathbb{P})$ is a linear map.
\item   \label{mixed-estimate}  For $g \in {\cal E}_k$ we have that
$$\EE I_{\mu_{1:k}}(g)_T =0 \quad \text{ if } |J|<k$$
and 
$$ \EE I_{\mu_{1:k}}(g)_T^2 \le \prod_{n \in J}  \m_n((0,T]\times \R)   \int_{((0,T]\times \R)^{k}} g^2(z_1,...,z_k) \m_1(dz_1) \cdots  \m_k (dz_k),$$
where $ d\m_n(z):= f_n^2(z)\m(dz).$ 
\item  \label{mixed-estimate-gen} Setting $ L^2_{\m_{1:k}}:=L^2(([0,T]\times\Rbb)^k, \mathcal{B}( ([0,T]\times\Rbb)^k), \otimes_{n=1}^k \m_n)$ one can extend $I_{\mu_{1:k}}$ to $ L^2_{\m_{1:k}}$ by linearity and continuity,
 and 
for any $f \in  L^2_{\m_{1:k}}$  we have
$$ \| I_{\mu_{1:k}}(f)_T\|_{L^2(\PP)} \le \prod_{n \in J}  \m_n((0,T]\times \R)^\frac{1}{2}  \,\, \|f\|_{ L^2_{\m_{1:k}}} .$$
\item \label{mixed-permutation}   Define for $f \in  L^2_{\m_{1:k}}$ and a permutation $\sigma \in {\tt S}_k$ the function  $$f^\sigma(z_1,...,z_k):=f(z_{\sigma(1)},...,z_{\sigma(k)}).$$
Then    $ I_{\mu_{\sigma(1):\sigma(k)}}(f)_T=I_{\mu_{1:k}}(f^\sigma)_T \quad a.s. $
\end{enumerate}
\end{lemma}

\begin{proof}
\eqref{mixed-linear} is clear. 
To see \eqref{mixed-estimate} we may assume  that $\mu_{n_1}, ... , \mu_{n_l}$ are random
and $\mu_{n_{l+1}},..., \mu_{n_k}$  are deterministic.
Since $M$ is an independent  random measure with expectation zero one can easily see that for disjoint $B_1,...,B_l \in \mathcal{B}( [0,T]\times \R)$ 
the random variables $\mu_{n_1}(B_1), ... , \mu_{n_l}(B_l)$ are independent and centered, and 
$\EE \mu_{n_i}(B_i)^2 = \m_i(B_i)$ holds for $i=1,...,l.$    
This implies  by H\"older's inequality, 
\equa
&&  \EE I_{\mu_{1:k}}(g)_T^2 \\
& =& \EE \Big ( \int_{((0,T]\times \R)^{k-l}}   \sum_{j_1,...,j_k=1}^n  a_{j_1,...,j_k} \mu_{n_1}(A_{j_{n_1}})\cdots \mu_{n_l}(A_{j_{n_l}})  \one_{A_{j_{n_{l+1}}}}(z_{l+1}) ...\one_{A_{j_{n_k}}}(z_k) \\
&& \hspace{20em} \times \m_{n_{l+1}}(dz_{l+1}) ...\m_{n_k}(dz_k) \Big )^2 \\
&\le& \prod_{\ell=l+1}^k  \m_{n_\ell}((0,T]\times \R) \\
&&\times \int_{((0,T]\times \R)^{k-l}} \hspace{-1em} \EE \left ( \sum_{j_1,...,j_k=1}^n  a_{j_1,...,j_k} \mu_{n_1}(A_{j_{n_1}})\cdots \mu_{n_l}(A_{j_{n_l}})  \one_{A_{j_{n_{l+1}}}}(z_{l+1}) ...\one_{A_{j_{n_k}}}(z_k)
\right )^2\\
&& \hspace{20em} \times \m_{n_{l+1}}(dz_{l+1}) ...\m_{n_k}(dz_k)  \\
&=& \prod_{\ell=l+1}^k  \m_{n_\ell}((0,T]\times \R) \\
&&\times \int_{((0,T]\times \R)^{k-l}} \hspace{-1em} \EE  \sum_{j_1,...,j_k=1}^n  a_{j_1,...,j_k}^2 \mu_{n_1}(A_{j_{n_1}})^2\cdots 
\mu_{n_l}(A_{j_{n_l}})^2  \one_{A_{j_{n_{l+1}}}}(z_{l+1}) ...\one_{A_{j_{n_k}}}(z_k)\\
&& \hspace{20em} \times \m_{n_{l+1}}(dz_{l+1}) ...\m_{n_k}(dz_k)  \\
&=& \prod_{\ell=l+1}^k  \m_{n_\ell}((0,T]\times \R)   \int_{((0,T]\times \R)^{k}} g^2(z_1,...,z_k) \m_1(dz_1) \cdots  \m_k (dz_k),
\tion
where we used that $\EE\mu_{n_1}(A_{j_{n_1}})^2\cdots 
\mu_{n_l}(A_{j_{n_l}})^2= \m_{n_1}(A_{j_{n_1}})^2\cdots 
\m_{n_l}(A_{j_{n_l}})^2$ for the last line. \\
\eqref{mixed-estimate-gen} One can show that ${\cal E}_k$ is dense in $ L^2_{\m_{1:k}}$ by adjusting the proof given in \cite[Section 1.1.2]{N06} for a product measure with equal components to our case where we have $\otimes_{n=1}^k \m_n.$  Then  the assertion follows from \eqref{mixed-estimate}. \\
\eqref{mixed-permutation}
For $f \in  {\cal E}_k$ we have
\equa
 f^\sigma(z_1,...,z_k) &=& \sum_{j_1,...,j_k=1}^n  a_{j_1,...,j_k}   \one_{A_{j_1}\times ...\times A_{j_k}}(z_{\sigma(1)},...,z_{\sigma(k)})\\
&=& \sum_{j_1,...,j_k=1}^n  a_{j_1,...,j_k}   \one_{A_{j_{\sigma^{-1}(1)}}\times ...\times A_{j_{\sigma^{-1}(k)}}}(z_1,...,z_k)
 \tion
which implies
\equa I_{\mu_{1:k}}(f^\sigma )
&=& \sum_{j_1,...,j_k=1}^n  a_{j_1,...,j_k}  \mu_1({A_{j_{\sigma^{-1}(1)}}) ... \mu_k(A_{j_{\sigma^{-1}(k)}}})  \\
&=& \sum_{j_1,...,j_k=1}^n  a_{j_1,...,j_k}  \mu_{\sigma(1)}(A_{j_1}) ... \mu_{\sigma(k)}(A_{j_k})  \\
&=& I_{\mu_{\sigma(1):\sigma(k)}}(f).
\tion
For general $f \in L^2_{\m_{1:k}}$  we use again the approximation argument.
\end{proof}

Now we can state a relation between our mixed iterated integrals and  mixed multiple integrals.

\begin{lemma}  \label{lemma-multiple}   Assume that
$$(\cM^{i_1}(dz_1),...,\cM^{i_k}(dz_k))=(   {\bf m} (dz_1), \ldots ,{\bf m} (dz_{n}), M(dz_{n+1}),...,  M(dz_{k})).$$    
If $h_1,...,h_n \in  L_1^1$ and $h_{n+1},...,h_k \in  L_1^2,$ then (recall $\Delta_{\,T}(k)$  given in \eqref{Delta}) we have
 \equa  
&& \sum_{\sigma \in {\tt S}_k} \int_{\Delta_{\,T}(k)}   h_{\sigma(1)}(z_{1}) \cdots h_{\sigma(k)}(z_k)  \cM^{i_{\sigma(1)}}(dz_1) \ldots       \cM^{i_{\sigma(k)}}( dz_k)\\
&&=\left ( \int_{((0,T]\times \R)^{n}}   h_1(z_1)\cdots h_{n}(z_{n})  \,\,
{\bf m} (dz_1) \ldots {\bf m} (dz_{n})\right ) \,\,  I_{k-n} \left ( \otimes_{i=1}^{k-n} h_{n+i}  \right)_T.
\tion
\end{lemma}
\begin{proof} We set 
$$  \mu_\ell(dz_\ell) := \left \{ \begin{array}{ll}  h_\ell(z_\ell) \m(dz_\ell) & \text{ if }  \, \,  \ell=1,...,n   \\ \\
h_\ell(z_\ell) M(dz_\ell) & \text{ if } \,\ \ell=n+1,...,k.\\\\
        \end{array}      \right . $$
For $f(z_1,...,z_k) := \one_{\Delta_{\,T}(k)}(z_1,...,z_k)$ it holds
\equa
 \sum_{\sigma \in \tt{S}_k}  f^\sigma (z_1,...,z_k) 
 &=&  \sum_{\sigma \in \tt{S}_k}   \one_{\Delta_{\,T}(m)}(z_{\sigma(1)},...,z_{\sigma(k)})  \\
 &=&  \sum_{\pi=\sigma^{-1} \in \tt{S}_k}   \one_{\{(t_{\pi(1)}, x_{\pi(1)},\ldots, t_{\pi(k)},   x_{\pi(k)}): \ 0\leq t_{\pi(1)}  <\ldots  < t_{\pi(k)} < T  \}}(z_1,...,z_k)\\
 &=& \one_{((0,T]\times \R)^k}(z_1,...,z_k)  \quad  \m^{\otimes k} a.e.
 \tion
Then, by Lemma \ref{mixed-properties} 
\begin{align*}
& \sum_{\sigma \in \tt{S}_k} \int_{\Delta_{\,T}(k)}   h_{\sigma(1)}(z_{1}) \cdots h_{\sigma(k)}(z_k)  \cM^{i_{\sigma(1)}}(dz_1) \ldots       \cM^{i_{\sigma(k)}}( dz_k)\\
=&  \sum_{\sigma \in \tt{S}_k}  \int_{((0,T]\times \R)^{k}} f(z_1,...,z_k) \mu_{\sigma(1)}(dz_1)...\mu_{\sigma(k)}(dz_k) \\
=& \sum_{\sigma \in \tt{S}_k}  I_{\mu_{\sigma(1):\sigma(k)}}(f)_T
=I_{\mu_{1:k}} \bigg( \sum_{\sigma \in \tt{S}_k}  f^\sigma \bigg)_T 
=   I_{\mu_{1:k}} \Big(  \one_{((0,T]\times \R)^k} \Big)  \\
=&\left ( \int_{((0,T]\times \R)^{n}}   h_1(z_1)\cdots h_{n}(z_{n})  \,\, 
{\bf m} (dz_1) \ldots {\bf m} (dz_{n})\right ) \,\,  I_{k-n} \left ( \otimes_{i=1}^{k-n} h_{n+i}  \right)_T.
\end{align*}
\end{proof}

\begin{proof}[Proof of Proposition \ref{alpha-products}] 
We have by  Proposition  \ref{product-iterated} and \eqref{random-deterministic}
\equa 
&&\prod_{j=1}^N J_{m_j}(\al_j^{\otimes m_j})_T  \\
&=&   \sum_k
\sum_{({\bf s}_1,...,{\bf s}_k)  \in A_k } \,\,   \sum_{{\bf i} \in \{0,1\}^{k}}  \int_{\Delta_{\,T}(k)}  \al_{{\bf s}_1,i_1}(z_1)\cM^{ i_1}(dz_1) \ldots       
\al_{{\bf s}_{k},i_k}(z_k)  \cM^{i_k}( dz_k).
\tion

We denote by $ [{\bf s}_1,i_1,...,{\bf s}_k,i_k]/\!\!\sim $ all equivalence classes  $[({\bf s}_1,i_1),...,({\bf s}_k,i_k)]$ with respect to permutations from ${\tt S}_k$
 where $({\bf s}_1,...,{\bf s}_k)  \in A_k$ and $(i_1,...,i_k)  \in \{0,1\}^{k}. $   
Our intention is to replace the above expression by  the following one (up to some factors which we want to determine next) 
$$ \sum_{k}       \sum_{ [{\bf s}_1,i_1,...,{\bf s}_k,i_k] /\!\sim \,\,} 
 \sum_{\sigma \in \mathtt{S}_k}
\ \int_{\Delta_{\,T}(k)} 
\prod_{j=1}^k \al_{{\bf s}_{\sigma(j)}, i_{\sigma(j)}}(z_j) \cM^{i_{\sigma(1)}}(dz_1) \ldots       \cM^{i_{\sigma(k)}}( dz_{k}).
$$ 
Whenever we have the equality (defining $\sim$)
\begin{align*}
(({\bf s}_1,i_1),...,({\bf s}_k,i_k))= (({\bf s}_{\sigma(1)},i_{\sigma(1)}),...,({\bf s}_{\sigma(k)},i_{\sigma(k)})),
\end{align*}
the same summand appears.   Like in the multinomial theorem, also here the multiplicity of the summands equals the number of those permutations that  do not change a tuple (in the multinomial theorem for $(a_1+\dotsb+a_K)^n$, with pairwise disjoint $a_i$, the multiplicity of the term   $\prod_{i=1}^K  a_i^{k_i} $   is $ \frac{n!}{k_1!\cdots k_K!},$ where $k_1+\dotsc+k_K=n$). 
 In our case, for any $({\bf s}_1,...,{\bf s}_k)  \in A_k$ and ${\bf i} \in \{0,1\}^{k}$, to count the multiplicity of the  ${\bf s}_m$ 
in $(({\bf s}_1,i_1),...,({\bf s}_k,i_k))$ paired  either with $i_m=1$ or $i_m=0$, we define for $j=1,..., \sn$ 
$$l_j :=| \{m: i_m = 1 \text{ and }  {\bf s}_m  =  {\bf e}_j\}|    \quad \text{and } \quad  
l^o_j :=| \{m: i_m = 0 \text{ and }  {\bf s}_m  =  {\bf e}_j\}|.$$
Then
\equa 
&&   \sum_k
\sum_{({\bf s}_1,...,{\bf s}_k)  \in A_k } \,\,   \sum_{{\bf i} \in \{0,1\}^{k}}  \int_{\Delta_{\,T}(k)}  \al_{{\bf s}_1,i_1}(z_1)\cM^{ i_1}(dz_1) \ldots       
\al_{{\bf s}_{k},i_k}(z_k)  \cM^{i_k}( dz_k)\\
&\!\!=\!\!& \sum_{k}       \sum_{ [{\bf s}_1,i_1,...,{\bf s}_k,i_k]/\sim  \, } \frac{ 1 }{ \,l!   \,l^o!   }
 \sum_{\sigma \in \mathtt{S}_k}
\ \int_{\Delta_{\,T}(k)} 
\prod_{j=1}^k \al_{{\bf s}_{\sigma(j)}, i_{\sigma(j)}}(z_j) \cM^{i_{\sigma(1)}}(dz_1) \ldots       \cM^{i_{\sigma(k)}}( dz_{k}) .
\tion

We denote by $n:= |l^o| = k - \sum_{j=1}^k i_j$ the number of zeros in ${\bf i} \in \{0,1\}^k$ and
choose a permutation $\pi \in  \mathtt{S}_k$  for which 
 $$  \cM^{ i_{\pi(1)}}(dz_1) \ldots       \cM^{i_{\pi(k)}}( dz_{k}) =({\bf m} (dz_1), \ldots ,{\bf m} (dz_{n}), M(dz_{n+1}),...,  M(dz_{k})).$$
Then we have  by Lemma \ref{lemma-multiple} that 

 \equa 
&&  \sum_{\sigma \in \mathtt{S}_k}
\ \int_{\Delta_{\,T}(k)} 
\prod_{j=1}^k \al_{{\bf s}_{\sigma(j)}, i_{\sigma(j)}}(z_j) \cM^{i_{\sigma(1)}}(dz_1) \ldots       \cM^{i_{\sigma(k)}}( dz_{k})  \\
&& = \int_{((0,T]\times \R)^{k}} 
\prod_{j=1}^k \al_{{\bf s}_{\pi(j)}, i_{\pi(j)}}(z_j) \, {\bf m} (dz_1) \cdots {\bf m} (dz_{n}) M(dz_{n+1}) \cdots  M(dz_{k}) \\
&& =  
 \int_{((0,T]\times \R)^n} 
     \bigotimes_{j=1}^n \, \al_{{\bf s}_{\pi(j)}, 0}(z_1,...,z_n)  \,\m(dz_1)... \m(dz_n)   \,\,\,   I_{k-n}\!\! \left ( \bigotimes_{j=n+1}^{k} \al_{{\bf s}_{\pi(j)}, 1}  \right )_T.  \tion
This yields
\equa 
&&\prod_{j=1}^N J_{m_j}(\al_j^{\otimes m_j})_T  \\
&=& \sum_{k}    \sum_{ [{\bf s}_1,i_1,...,{\bf s}_k,i_k]/\sim  \, } \frac{ 1 }{ \,l!   \,l^o!   }  
 \sum_{\sigma \in \mathtt{S}_k}
\ \int_{\Delta_{\,T}(k)} 
\prod_{j=1}^k \al_{{\bf s}_{\sigma(j)}, i_{\sigma(j)}}(z_j) \cM^{i_{\sigma(1)}}(dz_1) \ldots       \cM^{i_{\sigma(k)}}( dz_{k})  \\
&=& \sum_{k}       \sum_{\begin{subarray}{c}|l^o|+|l|=k,\\ (l,l^o) \in D_N \end{subarray}}   \frac{ 1}{ \,l!   \,l^o!    } 
 \left (\int_{((0,T]\times \R)^{|l^o|} }   \al^{ \otimes l^o}d\m^{\otimes |l^o|} \right )\,\,  I_{|l|} (\al^{\otimes l})_T, 
\tion
where 
$$    \al^{\otimes l} = \al_{{\bf e}_1,1}^{\otimes l_1} \otimes ... \otimes\al_{{\bf e}_{\ssn},1}^{\otimes l_{\ssn}} \quad \text{and} \quad
 \al^{ \otimes l^o} = \al_{{\bf e}_1,0}^{\otimes l^o_1}\otimes ... \otimes\al_{{\bf e}_{\ssn},0}^{\otimes l^o_{\ssn}}.$$
 Rewriting the condition in  $A_k$ from Proposition \ref{product-iterated} to this setting leads to the set  $D_N$ given in \ref{Dsn}.
By rearranging the summands we get 
\equa 
\prod_{j=1}^N J_{m_j}(\al_j^{\otimes m_j})_T  
 = \sum_{k}       \sum_{\begin{subarray}{c}|l|=k,\\ (l,l^o) \in D_N \end{subarray}}   \frac{ 1}{ \,l!   \,l^o!    } 
 \left (\int_{((0,T]\times \R)^{|l^o|} }   \al^{ \otimes l^o}d\m^{\otimes |l^o|} \right )\,\,  I_{k} (\al^{\otimes l})_T.
\tion
Finally, to get \eqref{first-product-formula}, we use Lemma \ref{iterated-and-multiple} to write the iterated integrals on the l.h.s.~as multiple integrals  which implies on the r.h.s.~the factor $m_1!\cdots m_N!.$
\end{proof}

\end{document}